\documentclass[a4paper,11pt,reqno]{amsart}

\usepackage{iftex}
\ifPDFTeX
  \usepackage[utf8]{inputenc}
  \usepackage[T1]{fontenc}
  \usepackage{lmodern}
\else
  \usepackage[no-math]{fontspec} 
\fi

\usepackage[hmargin=2.5cm,vmargin=2.5cm]{geometry}
\usepackage[dvipsnames]{xcolor}
\usepackage{enumitem}
\usepackage{graphicx}
\usepackage{tikz}
\usetikzlibrary{calc,decorations.pathreplacing,patterns,arrows.meta}
\usepackage{pgfplots}
\pgfplotsset{compat=1.18}

\usepackage{csquotes}
\usepackage[backend=biber,
  style=alphabetic,
  maxbibnames=99,
  doi=false, url=false, giveninits=true, date=year, doi=false]{biblatex}
\AtEveryBibitem{\clearfield{issn}}
\usepackage{macros}

\usepackage[colorlinks, linkcolor={blue}, citecolor={red}]{hyperref}

\newcommand{\sprod}[2]{\langle #1, #2 \rangle_{\C^N}} 

\begin{document}
\title[The Keller problem for Dirac operators with Lorentz scalar potentials]{The Keller problem for the lowest positive eigenvalue of Dirac operators with Lorentz scalar potentials}
\date{\today}

\author{David Gontier}
\address[David Gontier]{CERMICS, CNRS, ENPC, Institut Polytechnique de Paris, Marne-la-Vall\'ee, France.}
\email{david.gontier@enpc.fr}

\author{Hanne Van Den Bosch}
\address[Hanne Van Den Bosch]{Departamento de Ingeniería Matemática and Center for Mathematical Modeling, Facultad de Ciencias Físicas y Matemáticas, Universidad de Chile and CNRS IRL 2807,
Beauchef 851, Piso 5, Santiago, Chile.}
\email{hvdbosch@dim.uchile.cl}


\begin{abstract}
    We study the Keller problem for the Dirac operator $\Dirac_m - V\beta$ in $\R^d$, and minimize its lowest positive eigenvalue over all potentials $V$ with prescribed $L^p$ norm. We prove that optimal potentials always exist and that the corresponding eigenspinors solve the nonlinear Soler equation. This provides a variational proof of existence of solitary waves for this equation. We prove that there are no zero modes for such operators in dimensions $1$ and $2$, and conjecture that zero mode may exist in dimension $3$. In dimension $d = 2$, we provide numerical evidence that the optimal potentials concentrate on flower-shape regions, which are not necessarily radial, showing that Faber--Krahn type results may fail for such Dirac operators. In dimension $1$, all quantities are explicit.
\end{abstract}

\maketitle

\addtocontents{toc}{\protect\setcounter{tocdepth}{-1}}
\tableofcontents
\addtocontents{toc}{\protect\setcounter{tocdepth}{2}}

\section{Introduction and main results}
\label{sec:intro}

The original problem raised by Keller~\cite{Kel-61} consists of minimizing the lowest eigenvalue of the Schr\"odinger operator $-\Delta - V$ over all nonnegative potentials $V$ with fixed $L^p$ norm, namely
\[
    \inf\bigl\{ \lambda_1(- \Delta - V), \quad V \in L^p(\R^d, \R^+), \  \|V\|_{L^p} = \alpha \bigr\}.
\]
It is nowadays a classical fact that when $p > d/2$, this minimization problem is well-posed, in the sense that it admits optimizers~\cite{LieThi-76}. The optimal potential is radially symmetric, and the associated optimal eigenfunction is, up to scaling and phase, the unique positive ground state of the nonlinear Schr\"odinger equation (see~\cite{Cof-72, Kwo-89, McL-93})
\[
    -\Delta \phi - |\phi|^{\frac{2}{p-1}} \phi = \lambda_1 \phi, \qquad \lambda_1 < 0.
\]

The analogue of Keller's problem for relativistic systems requires more care, since the free Dirac operator $\Dirac_m = - \ri \balpha \cdot \nabla + m\beta$ has spectrum $(-\infty, -m] \cup [m, +\infty)$ (see below for details), and therefore is not bounded from below. A perturbation of the form $-V\bbI_N$ can create eigenvalues in the gap $(-m, m)$, and the relevant quantity is the lowest eigenvalue {\em in the gap}. We studied the Keller problem for $\Dirac_m - V$ in~\cite{DolGonPiz-23} (see also~\cite{AlaLi-92, AlaLi-92a} for related problems in the gap of general operators), and proved that this problem is well-posed whenever $p > d$, and that the corresponding eigenspinor, after a suitable scaling, solves the nonlinear Dirac equation with Kerr nonlinearity
\[
    \Dirac_m \psi - | \psi |^{\frac{2}{p-1}} \psi = \lambda \psi, \qquad -m < λ < m.
\]

\medskip

The goal of this article is to study the Keller problem for the Soler--Dirac operator. In this case, the free Dirac operator is perturbed by $-V\beta$, leading to the operator
\[
    \Dirac_m - V \beta =  - \ri \balpha \cdot \nabla + (m - V) \beta.
\]
The potential $V \in L^p(\R^d, \R)$ is therefore seen as a perturbation of the mass $m > 0$. This type of coupling appears in the classical models of self-interacting fermions of Ivanenko~\cite{Iva-38} and Soler~\cite{Sol-70} (see also~\cite{GroNev-74} in dimension $d=1$), in which the self-interaction depends on the Lorentz scalar $\overline{\psi}\psi := \sprod{\psi}{\beta \psi}$. In dimensions $d ∈ \{ 1, 2, 3\}$, the operator $\Dirac_m - V\beta$ enjoys extra symmetries, which implies that its spectrum is symmetric with respect to the origin (see Proposition~\ref{prop:symmetries} and Remark~\ref{rem:prob_sym} below). For these dimensions, the spectrum is of the form
\[
    \begin{cases}
    \sigma_{\rm ess}( \Dirac_m - V \beta) &  = (- \infty, -m] \cup [m, + \infty) , \\
    \sigma_{\rm disc} (\Dirac_m - V \beta) & = \{ ± \lambda_1, ± \lambda_2, \cdots \}, \quad 0 \le \lambda_1 \le \lambda_2 \le \cdots < m,
    \end{cases}
\]
where the sequence $(\lambda_n)$ is either finite or accumulates at $m$. In particular, there is no {\em lowest} eigenvalue in the gap $(-m, m)$. Our goal is rather to optimize the {\bf eigenvalue closest to zero} of $\Dirac_m - V \beta$, which also corresponds to the lowest {\em positive} eigenvalue in dimensions $d ∈ \{ 1, 2, 3 \}$. We define
\begin{equation} \label{eq:def:lambda1}
    \lambda_1(V) := \dist\bigl(0, \sigma(\Dirac_m - V\beta)\bigr) = \inf \sigma \bigl( | \Dirac_m - V \beta | \bigr) \quad \in [0, m],
\end{equation}
and study the Keller problem
\begin{equation} \label{eq:def:Lambda_intro}
   \boxed{ \Lambda(\alpha) := \inf\Big\{\lambda_1(V), \quad V \in L^p(\R^d, \R), \quad \nrm{V}{L^p} \le \alpha \Big\}.}
\end{equation}

Since we optimize under the inequality $\nrm{V}{L^p} \le \alpha$, the map $\alpha \mapsto \Lambda(\alpha)$ is nonincreasing. Throughout the paper, we fix $d \ge 1$ and $p > d$, and we allow $V$ to change sign. The eigenvalue $\lambda_1(V)$ vanishes if and only if $\Dirac_m - V\beta$ has a {\em zero mode}, that is a nontrivial solution of $(\Dirac_m - V\beta)\psi = 0$. This leads us to introduce the \emph{critical mass}
\begin{equation} \label{eq:def:alpha_star_intro}
    \alpha_*(p,d) := \inf \bigl\{ \alpha > 0, \quad \Lambda(\alpha) = 0 \bigr\} \quad \in [0, +\infty].
\end{equation}
In the case where $\Lambda(\alpha) > 0$ for all $\alpha$, we set $\alpha_*(p,d) = + \infty$. Our first main result is the following.

\begin{theorem}[Existence of optimal potentials] \label{thm:existence}
    Let $d \ge 1$ and $p > d$. The map $\alpha \mapsto \Lambda(\alpha)$ is continuous and nonincreasing on $(0, \infty)$, with
    \[
        \lim_{\alpha \to 0} \Lambda(\alpha) = m \qquad \text{and} \qquad \lim_{\alpha \to \alpha_*} \Lambda(\alpha) = 0.
    \]
    This map is positive and strictly decreasing on $(0, \alpha_*(p,d))$. For all $\alpha > 0$, there is an optimal potential $V \in L^p(\R^d, \R)$ with $\nrm{V}{L^p} \le \alpha$ such that $\lambda_1(V) = \Lambda(\alpha)$. \\
    In addition, in the case $\alpha < \alpha_*(p,d)$, any optimal potential is bounded and satisfies $\nrm{V}{L^p} = \alpha$, and any associated normalized eigenspinor $\psi$ satisfies, for some $\mu = \mu(\psi) > 0$,
    \begin{equation} \label{eq:EL_intro}
        (\Dirac_m - V \beta) \psi = \Lambda(\alpha) \psi \qquad \text{and} \qquad \mu \, | V |^{p-2} V = \sprod{\psi}{\beta \psi}.
    \end{equation}
\end{theorem}

We display in Figure~\ref{fig:Lambda_alpha} some plots of $\Lambda(\alpha)$ for $d = 1$ and $d = 2$ and for different values of $p$.

\begin{figure}[ht]
\centering
\includegraphics[width=\linewidth]{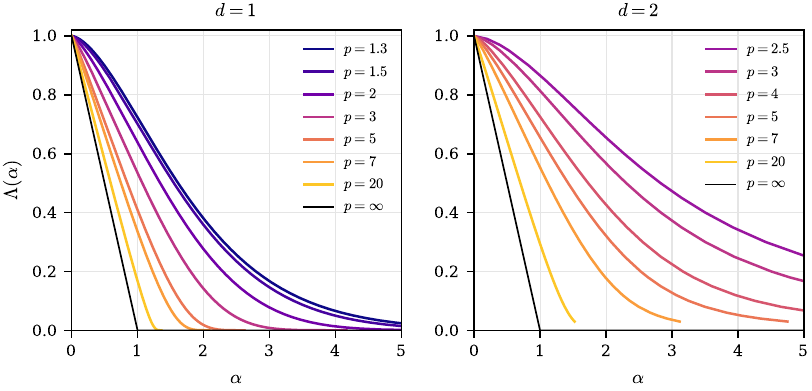}
\caption{The function $\alpha \mapsto \Lambda(\alpha)$ for different values of $p$. Left: $d = 1$ (explicit formulas, see Appendix~\ref{sec:appendix:explicit_1D}). Right: Same for $d = 2$, see Appendix~\ref{sec:appendix:numerics} for details on the numerics. In both cases, as $p \to \infty$, the curves approach the line $\Lambda = (1 - \alpha)_+$.}
\label{fig:Lambda_alpha}
\end{figure}

The function $\phi := \mu^{-1/2}\psi$ solves the nonlinear Soler equation
\begin{equation} \label{eq:Soler}
    \boxed{ \Dirac_m \phi - f\big( \sprod{\phi}{\beta \phi} \big) \beta \phi = \Lambda(\alpha) \phi \qquad \text{with} \qquad f(s) := |s|^{\frac{2-p}{p-1}} s.}
\end{equation}
When the optimal potential is nonnegative (this holds in dimension $d = 1$, see Appendix~\ref{sec:appendix:explicit_1D}), then $f(s) = s^{1/(p-1)}$ is the usual power-type Soler nonlinearity. 

\medskip

Solitary waves of the Soler model have been extensively studied, both analytically and numerically~\cite{CueBouCom-18, AldRicSto-23}. In dimension $d = 3$, existence of solutions is proved in~\cite{CazVaz-86}, using a shooting method in the Wakano ansatz~\cite{Wak-66}. This result was extended to more general nonlinearities in~\cite{Mer-88}, and the existence of infinitely many solutions (excited states) was proved in~\cite{BalCazDou-88a, BalCazDou-88}. In~\cite{EstSer-95}, solutions are obtained by linking arguments. Our method gives a proof of existence based on a direct variational argument: we minimize an eigenvalue. We prove existence of solutions for all frequencies $\Lambda \in (0, m)$, and without any symmetry assumption. This approach is novel to the best of our knowledge.

\medskip

We provide some numerical simulations in Appendix~\ref{sec:appendix:numerics} for $d = 2$. In view of the results, we observe that the optimal potentials we find are not necessarily symmetric, so our solutions differ from the ones found in the previous works, which assumed the Wakano ansatz (which implies that $\langle \psi, β ψ \rangle_{\C^N}$ is radial). We conjecture that, should one optimize among radially symmetric potentials only, we would probably recover the {\em ground state} solutions of the previous works.

\medskip

Our second main result concerns the critical mass $\alpha_*(p,d)$. As we prove below, if $\alpha_*(p,d) < \infty$, then $\alpha_*(p, d)$ is attained: it is the smallest $L^p$ norm of a potential creating a zero mode. The following result is proved in Section~\ref{sec:no_zero_modes}.
\begin{theorem}[Absence of zero modes in dimension $d = 1$ and $d = 2$] 
    \label{thm:critical_alpha_dim12} ~
    If $d=1$ or $d=2$, then for all $p > d$, we have $\alpha_*(p, d) = + \infty$. In other words, for all potential $V ∈ L^p(\R^d, \R)$ we have $0 \notin \sigma(\Dirac_m - V \beta)$.
\end{theorem}

In dimension $d = 3$, we rather conjecture the following
\begin{conjecture}[Existence of zero mode in dimension $d = 3$]
    \label{thm:critical_alpha_dim3}
    For all $p > 3$, we have $\alpha_*(p, 3) < +\infty$: there are potentials $V \in L^p(\R^3)$ creating zero modes, $0 \in \sigma(\Dirac_m - V\beta)$.
\end{conjecture}

ChatGPT Astra seems to have found a way to construct such a 3-d potential supporting zero modes. But since we do not yet fully understand its proof, and since we wanted to keep the present article free of AI generated proofs, we prefer to name this result a {\em conjecture}, and postpone the construction of such potentials for a future manuscript. 

\medskip

Conjecture~\ref{thm:critical_alpha_dim3} is quite surprising, in particular in view of the following result.
\begin{proposition} \label{prop:radialV}
    For all $d \ge 1$ and all $p > d$, if $V \in L^p(\R^d, \R)$ is {\bf radial}, then $0 \notin \sigma(\Dirac_m - V\beta)$.
\end{proposition}
In other words, a radial potential cannot create zero modes. If Conjecture~\ref{thm:critical_alpha_dim3} holds, then any optimal potential at $\alpha = \alpha_*(p,3)$, which creates a zero mode, is non-radial. To our knowledge, this would be the first example of a natural operator for which the optimal potential for the first eigenvalue is not radial. Our numerical simulations (Appendix~\ref{sec:appendix:numerics}) indicate that symmetry breaking could also occur in dimension $d = 2$, although there are no zero modes in this case. This is in stark contrast with the Faber--Krahn inequality in bounded domains for Dirac operators~\cite{Dur-26, DurMasSan-26, AmmQiu-26}, where the two-dimensional disk is the optimal shape for the lowest eigenvalue, at least for some natural boundary conditions (except in dimension $d = 1$, where all computations are explicit).
\medskip

Let us emphasize that the proof of Theorem~\ref{thm:existence} differs strongly from the one of~\cite{DolGonPiz-23}, although some key computations are similar. The key ingredient in~\cite{DolGonPiz-23} is a Birman--Schwinger argument: we fix $\lambda \in (-m, m)$ and recast the Keller problem as an optimization of the highest eigenvalue of the corresponding Birman--Schwinger operator, namely
\[
    \sup \left\{ \mu_1 \left(\sqrt{V} (\Dirac_m  - \lambda)^{-1} \sqrt{V} \right) , \qquad V \in L^p(\R^d, \R^+), \ \| V \|_{L^p} = \alpha \right\}.
\]
In our case, the corresponding Birman--Schwinger operator involves the operator
\begin{equation} \label{eq:def:KV}
    K_V(z) := |V|^{1/2}  (\Dirac_m - z)^{-1}  |V|^{1/2}  \sign(V) \beta,
\end{equation}
which is compact, but not self-adjoint because of the $\beta$ matrix. We cannot define its {\em highest} eigenvalue as we did in~\cite{DolGonPiz-23}. Our strategy is therefore to work directly with the eigenvalue $\lambda_1(V)$, using the characterization
\[
    \lambda_1(V) = \inf \left\{  \| (\Dirac_m - V\beta)ψ \|_{L^2}, \quad \| ψ \|_{L^2} = 1 \right\}.
\]
This formulation allows to adapt standard concentration-compactness results~\cite{Lio-84}.

\section{The Dirac operator with Lorentz scalar potentials}
\label{sec:prelim}

\subsection{First properties of the Dirac operator}
Let us first recall some basic properties of the Dirac operator (see also~\cite{Tha-92}). Let $d\ge 1$ and set $N :=2^{\lfloor(d+1)/2\rfloor}$. We consider $\alpha_1, \cdots, \alpha_d$ and~$\alpha_0 := \beta$, which are $N \times N$ Hermitian matrices satisfying the anti-commutation rules (Clifford relations)
\[
\forall j,k \in \{ 0, \cdots, d \}, \qquad    \alpha_j \alpha_k+\alpha_k \alpha_j=2 \delta_{jk} \bbI_N,
\]
where $\delta_{jk}$ denotes the Kronecker symbol and $\bbI_N$ is the $N\times N$ identity matrix. One way to construct these matrices is by induction, as follows. We denote by $\sigma_1, \sigma_2, \sigma_3$ the usual Pauli matrices. For $d = 1$ and $N = 2$, we set $\alpha_1 := \sigma_1$ and $\beta := \sigma_3$. For $d = 2$ and $N = 2$, we set $\alpha_1 := \sigma_1$, $\alpha_2 := \sigma_2$ and $\beta := \sigma_3$. Then, to go from spatial dimension $d$ to $d+2$, we set
\[
        \forall j \in \{ 0, \cdots, d\}, \quad \alpha_j^{(d+2)} := \alpha_j^{(d)} \otimes \sigma_3, \quad
        \alpha_{d+1}^{(d+2)} := \bbI_N \otimes \sigma_1, \quad 
        \alpha_{d+2}^{(d+2)} := \bbI_N \otimes \sigma_2
\]
(in particular, $\beta^{(d+2)} = \beta^{(d)} \otimes \sigma_3$). This gives $d+3$ matrices of size $2N \times 2N$ satisfying the Clifford relations. 

\medskip

\subsubsection{Self-adjointness.} The \emph{free Dirac operator} in dimension~$d$ is the operator acting on $L^2(\R^d, \C^N)$ defined by (we choose units in which $\hbar=c=1$)
\[
    \Dirac_m:=\sum_{j=1}^d\alpha_j (- \ri \partial_j)+m \beta =: \balpha \cdot (- \ri \nabla)+m \beta.
\]
It is self-adjoint with domain $H^1(\R^d, \C^N)$, and satisfies $\Dirac_m^2=-\Delta+m^2$. Its spectrum is purely essential, and given by
\[
    \sigma(\Dirac_m)=\sigma_{\rm ess}(\Dirac_m)=(-\infty,-m]\cup[m,+\infty).
\]
In what follows, we perturb the constant mass $m$ by a {\em Lorentz scalar} potential $V \in L^p(\R^d, \R)$, and consider the Dirac operator
\[
    H_V := \Dirac_m - V \beta = \balpha \cdot (- \ri \nabla)+(m - V) \beta.
\]

\begin{proposition}[Self-adjointness] 
    \label{prop:selfadjoint}
    Let $d \ge 1$, $p > d$ and $V \in L^p(\R^d, \R)$. The operator $\Dirac_m - V\beta$ has a distinguished self-adjoint extension, with domain
    \[
        \Dom(H_V) = \Big\{ \psi \in L^2(\R^d, \C^N), \quad |V|^{1/2} \psi \in L^2 \quad \text{and} \quad (\Dirac_m - V\beta) \psi \in L^2 \Big\}.
    \]
    We always have $\Dom(H_V) \subset H^{1/2}(\R^d, \C^N)$. If $p \ge 2$, it holds $\Dom(H_V) = H^1(\R^d, \C^N)$ and 
    \[
        \sigma_{\rm ess} \left( \Dirac_m - V \beta \right) = \sigma_{\rm ess} \left( \Dirac_m \right) = (-\infty,-m]\cup[m,+\infty).
    \]
    The Birman--Schwinger principle holds: $\lambda \in (-m, m)$ is an eigenvalue of $\Dirac_m - V\beta$ if and only if $1$ is an eigenvalue of the Birman--Schwinger operator $K_V(\lambda)$.
\end{proposition}

The proof follows the exact same lines as~\cite[Proposition 2.1]{DolGonPiz-23}, which relies on the method of Kato and Nenciu~\cite{Kat-66, Nen-76}. The Kato--Seiler--Simon inequality~\cite[Theorem 4.1]{Sim-05} shows that $K_V(z)$ is in the Schatten class $\fS_p$ whenever $p \ge 2$. Under the same assumption $p \ge 2$, Sobolev's inequalities show that $V\beta$ is $\Dirac_m$-bounded with relative bound $0$. This implies that the domain is $H^1(\R^d, \C^N)$ by the Kato--Rellich theorem~\cite{Rel-37, Kat-51}. In the case $d = 1$ and $1 < p < 2$, the domain is not $H^1(\R, \C^2)$ in general.

\medskip

\subsubsection{Symmetries.} As we already mentioned, the main difference between the Dirac with a Lorentz scalar potential $V \beta$ and the Dirac with a potential $V$ is the presence of extra symmetries. An anti-unitary operator $C$ satisfying $C H_V = - H_V C$ is called a {\bf charge-conjugation symmetry}. It is bosonic if $C^2 = +1$ and fermionic if $C^2 = -1$. Similarly, an anti-unitary operator $T$ satisfying $T H_V = H_V T$ is called a {\bf time-reversal symmetry}. It is bosonic if $T^2 = +1$ and fermionic if $T^2 = -1$. In the presence of both $T$ and $C$ symmetries, we assume that $S = TC$ satisfies $S^2 = +1$. This is a unitary operator satisfying $S H_V= - H_V S$, and called a {\bf S-symmetry}.

\begin{proposition}[Symmetries] 
    \label{prop:symmetries}
    Let $d \ge 1$. The existence of $T$ and/or $C$-symmetries depends on $d$ modulo $8$, as follows.
    \begin{table}[ht]
    \centering
    \renewcommand{\arraystretch}{1.3}
    \begin{tabular}{|c||c|c|c|c|c|c|c|c|}
        \hline
        $d \bmod 8$ & $1$ & $2$ & $3$ & $4$ & $5$ & $6$ & $7$ & $0$ \\
        \hline\hline
        $C^2$ & $+1$ & $+1$ & $+1$ & -- & $-1$ & $-1$ & $-1$ & -- \\
        \hline
        $T^2$ & $+1$ & -- & $-1$ & $-1$ & $-1$ & -- & $+1$ & $+1$ \\
        \hline
        $S$ & yes & -- & yes & -- & yes & -- & yes & -- \\
        \hline
        class & {\rm BDI} & {\rm D} & {\rm DIII} & {\rm AII} & {\rm CII} & {\rm C} & {\rm CI} & {\rm AI} \\
        \hline
    \end{tabular}
    \label{tab:symmetries}
\end{table}

\medskip

    As a consequence, if $d \not\equiv 0 \pmod 4$, then 
    \[
    C (\Dirac_m - V \beta) C^{-1} = - (\Dirac_m - V \beta), \quad \text{hence} \quad \sigma(\Dirac_m - V \beta) = - \sigma(\Dirac_m - V \beta),
    \]
    and the spectrum is symmetric with respect to the origin. Moreover, if $d \pmod 8 ∈ \{ 3, 4, 5 \}$, then all eigenvalues have even multiplicities.
\end{proposition}

The classes are the Altland--Zirnbauer classes~\cite{AltZir-97, HeiHucZir-05} (see also~\cite{GonMonPer-22, GonTau-25}), whose names come from Cartan's classification of symmetric spaces~\cite{Car-26, Car-27}. It is interesting to note that, as $d$ varies, the massive Dirac operator model spans the $8$ Altland--Zirnbauer classes, in the order dictated by Bott periodicity~\cite{Bot-56, AtiBotSha-64}, as in the periodic table of topological insulators~\cite{Kit-09, RyuSchFur-10}. 

\begin{proof}
    We denote by $\cK : f \mapsto \overline{f}$ the entry-wise complex conjugation. We say that a unitary matrix $U$ is of type $\pm$ if
    \[
        \forall j \in \{1, \cdots, d\}, \quad U \, \overline{\alpha_j} \, U^* = \mp \alpha_j, \qquad \text{and} \qquad U \, \overline{\beta} \, U^* = \pm \beta.
    \]
    Since $\cK ( - \ri \partial_j) \cK = + \ri \partial_j$, we easily check that if $U$ is of type $+$, then $T := U \cK$ is a $T$-symmetry, and if $U$ is of type $-$, then $C := U \cK$ is a $C$-symmetry.

    \medskip
\noindent\textbf{ Existence:}
    In dimension $d = 1$, we took $\alpha_1 = \sigma_1$ and $β = \sigma_3$. Since $\overline{\sigma_1} = \sigma_1$ and $\overline{\sigma_3} = \sigma_3$, we deduce that $\sigma_1$ is of type $-$ while $\sigma_3$ is of type $+$. The corresponding anti-unitaries $C = \sigma_1 \cK$ and $T = \sigma_3 \cK$ satisfy $C^2 = +1$ and $T^2 = +1$. In dimension $d = 2$, we took $\alpha_1 = \sigma_1$, $\alpha_2 = \sigma_2$ and $β = \sigma_3$, and since $\overline{\sigma_2} = -\sigma_2$, we deduce that $\sigma_1$ is of type $-$ in dimension $d = 2$. The corresponding $C$ symmetry satisfies $C^2 = +1$.
    
    \medskip
    
    Following the construction of the $\alpha_j$ matrices in the previous section, one checks that if $U$ is of type $+$ in dimension $d$, then $U \otimes \sigma_1$ is of type $-$ in dimension $d+2$, satisfying $(U ⊗ \sigma_1 \cK)^2 = (U \cK)^2\otimes \bbI$. On the other hand, if $U$ is of type $-$, then $U \otimes \sigma_2$ is of type $+$, with  $(U ⊗ \sigma_2 \cK)^2 = - (U \cK)^2 \otimes \bbI$. In other words, a fermionic (resp. bosonic) $C$ symmetry in dimension $d$ gives rise to a bosonic (resp. fermionic) $T$ symmetry in dimension $d + 2$ and a fermionic (resp. bosonic) $T$ symmetry in dimension $d$ gives rise to a fermionic (resp. bosonic) $C$ symmetry in dimension $d+2$, as can be read off from the table. 

    \medskip
\noindent\textbf{ Multiplicity:}
    In the case where the operator enjoys a fermionic $T$--symmetry, all eigenvalues of $\Dirac_m - V\beta$ have even multiplicity by Kramers' argument~\cite{Kra-30} (if $(H_V - \lambda) ψ = 0$ then the vector $ϕ := T \psi$ satisfies $(H_V - \lambda) ϕ = 0$, and $⟨ \phi, ψ ⟩ = 0$).
    \medskip
    
\noindent\textbf{ Non-existence:}   
  When $d \equiv 0 \pmod 4$, there are no operators $U$ of type $-$, hence no $C$-symmetry. To see this, note that the matrix $\omega := \alpha_1 \cdots \alpha_d \beta$ would satisfy $U \overline{\omega} U^{*} = - \omega$ for $U$ of type $-$. On the other hand a direct induction shows that $ω = (-\ri)^{d/2} \bbI_N$ for $d$ even, which is real if $d \equiv 0 \pmod 4$, so that $U \overline{\omega} U^* = \omega$, a contradiction.  
  The same argument shows that there are no operators $U$ of type $+$ if $d  \equiv 2 \pmod 4$. 
\end{proof}

\begin{remark} \label{rem:prob_sym}
At this point, what we minimize is $\lambda_1(V)$, which is the lowest eigenvalue {\em in magnitude}. It is unclear a priori that it indeed corresponds to the lowest {\em positive} eigenvalue, as claimed in the title. It is the case when $d \not≡ 0 \pmod 4$ by symmetry. When $d ≡ 0 \pmod 4$, the discrete spectrum of $H_V := \Dirac_m - Vβ $ is not necessarily symmetric with respect to the origin, and of the form
\[
\sigma_{\rm disc}(H_V) = \{ \lambda_1^+, \lambda_2^+, \cdots  \} ∪ \{ \lambda_1^-, \lambda_2^-, \cdots  \}, \qquad
0 \le \lambda_1^+ \le \lambda_2^+ \le \cdots , \quad 0 \ge \lambda_1^- \ge \lambda_2^- \ge \cdots,
\]
and we minimize $\lambda_1(V) := \min \{ | \lambda_1^+| , | \lambda_1^-| \}$.
Even though the spectrum for an individual operator $\Dirac_m -V\beta$ is not necessarily symmetric, the unitary operator $\alpha_1 P_1$, where $P_1$ is a reflection in the first coordinate, anticommutes with $\Dirac_m$ and transforms $V$ into $\widetilde V$, its reflection with respect to the plane $x_1 = 0$.  
Hence $\lambda^{\pm}_1(V) = - \lambda^{\mp}_1(\widetilde V)$ and $\inf_V  \lambda_1^+(V) = \inf_V | \lambda_1^-(V) | = \inf_V \lambda_1(V)$.
\end{remark}

\subsection{First bounds of \texorpdfstring{$\lambda_1(V)$}{lambda1(V)}}

Recall that our goal is to minimize the quantity
\[
    \lambda_1(V) = \inf \left\{  \| (\Dirac_m - V\beta)ψ \|_{L^2}, \quad ψ ∈ \Dom(H_V), \quad \| ψ \|_{L^2} = 1 \right\}.
\]
Let us first state some simple bounds on this first eigenvalue.
\begin{lemma}[A priori bounds] \label{lem:bounds_lambda1}
    Let $d \ge 1$ and $p > d$.
    \begin{enumerate}
        \item (Upper bound) For all $V ∈ L^p(\R^d, \R)$, we have
        \[
            \lambda_1(V)^2 \le \inf \sigma \big( -\Delta + (m - V)^2 \big).
        \]
        \item (Lower bound) There is $c_{d,p} > 0$ so that, for all $V ∈ L^p(\R^d, \R)$, we have
        \[
            \lambda_1(V) \ge m - c_{d,p} m^{d/p} \| V \|_{L^p}, \qquad \text{hence} \qquad \Lambda(\alpha) \ge m - c_{d,p} m^{d/p} \alpha.
        \]
        If in addition $V ∈ L^\infty(\R^d, \R)$, then $\lambda_1(V) \ge m - \| V \|_{L^\infty}$.
    \end{enumerate}
    
\end{lemma}

\begin{proof}
Let us first prove the upper bound. The square of the Dirac operator is formally the matrix-valued Schrödinger operator
\begin{equation} \label{eq:square}
    (\Dirac_m - V \beta)^2 = - \Delta + (m - V)^2 + \ri \, (\balpha \cdot \nabla V) \beta.
\end{equation}
Consider a test function of the form $ψ = f \be$, where $f ∈ C^\infty_0(\R^d, \R)$ and $\be ∈ \R^N$ is a fixed normalized vector, independent of $x$. Then
\[
    ⟨ ψ, (\Dirac_m - V \beta)^2 ψ ⟩ = \int_{\R^d} | ∇ f |^2 + \int_{\R^d} ( m - V)^2 | f |^2 + \sum_{j=1}^d ⟨ \be, \alpha_j β \be ⟩_{\C^N}  \int_{\R^d} V (\ri \partial_j) (| ψ |^2).
\]
We choose $\be$ so that $β \be = \be$. In particular, we have 
\[
   ⟨ \be, \alpha_j \be ⟩ = ⟨ \be, \alpha_j β \be ⟩ = ⟨ \be, - β \alpha_j \be ⟩ = - ⟨ \be, \alpha_j \be ⟩,
\]
so all quantities $⟨ \be, \alpha_j β \be ⟩_{\C^N}$ vanish. Minimizing over $f ∈ C^\infty_0(\R^d, \R)$ shows the upper bound.

\medskip

To prove the lower bound, we use the Birman--Schwinger principle. Consider $0 \le \lambda < m$ (the proof for $-m < λ \le 0$ is similar). The operator $(\Dirac_m - \lambda)^{-1}$ is a Fourier multiplier by the matrix $(\balpha \cdot \bk + m \beta - \lambda)^{-1}$, whose eigenvalues are $(\pm \sqrt{| \bk |^2 + m^2} - \lambda)^{-1}$. From the inequalities 
\[
    \left| \dfrac{1}{\balpha \cdot \bk + m \beta - \lambda} \right| \le \dfrac{1}{\sqrt{| \bk |^2 + m^2} - \lambda} \le \dfrac{m}{m - \lambda} \dfrac{1}{\sqrt{\bk^2 + m^2}},
\]
we get the operator bound $|(\Dirac_m - \lambda)^{-1}|  \le \frac{m}{m - \lambda} (- Δ + m^2)^{-1/2}$. This gives
\begin{align*}
    \left\|K_V(\lambda) \right\|_{\rm op} & = \left\| |V|^{1/2} (\Dirac_m - \lambda)^{-1} |V|^{1/2} \right\|_{\rm op} 
    \le \frac{m}{m - \lambda} \left\| |V|^{1/2} ( - \Delta + m^2)^{-1/2} |V|^{1/2} \right\|_{\rm op} \\
    & = \frac{m}{m - \lambda} \left\| |V|^{1/2} ( - \Delta + m^2)^{-1/4} \right\|_{\rm op}^2 
    \le \frac{m}{m - \lambda} \left\| |V|^{1/2} ( - \Delta + m^2)^{-1/4} \right\|_{\fS_{2p}}^2 
\end{align*}
To bound the last norm, we use the Kato--Seiler--Simon inequality~\cite[Theorem 4.1]{Sim-05}: for all $r \ge 2$ and all $f, g \in L^r(\R^d)$, the operator $f(x) g(-\ri \nabla)$ belongs to the Schatten class $\fS_r$, with
\begin{equation*} 
    \nrm{f(x) \, g(-\ri \nabla)}{\fS_r} \le \frac{1}{(2 \pi)^{d/r}} \nrm{f}{L^r} \nrm{g}{L^r}.
\end{equation*}
We obtain, with a scaling argument, a bound of the form
\[
    \left\|K_V(\lambda) \right\|_{\rm op} \le \dfrac{m^{d/p}}{m - \lambda} c_{d,p} \| V \|_{L^p}
\]
with the explicit constant
\[
    c_{d,p} := \frac{1}{(2\pi)^{d/p}} \Big( \int_{\R^d} \frac{\rd \bk}{(1 + |\bk|^2)^{p/2}} \Big)^{1/p} = \frac{1}{(4 \pi)^{\frac{d}{2p}}} \left( \dfrac{Γ \left( \frac{p-d}{2} \right)}{Γ \left( \frac{p}{2} \right)} \right)^{\frac{1}{p}}.
\]
The Birman--Schwinger principle states that $\lambda$ is an eigenvalue of $\Dirac_m - V \beta$ if and only if $1$ is an eigenvalue of $K_V(\lambda)$. This implies in particular that $\| K_V(\lambda) \|_{\rm op} \ge 1$, and can only happen if
\[
    m^{d/p} c_{d,p} \| V \|_{L^p} \ge m - \lambda.
\]

Finally, if $V \in L^\infty(\R^d, \R)$, then for all $\psi \in H^1(\R^d, \C^N)$, we have
\[
    \| (\Dirac_m - V\beta) \psi \|_{L^2} \ge \| \Dirac_m \psi \|_{L^2} - \| V \beta \psi \|_{L^2} \ge \big( m - \nrm{V}{L^\infty} \big) \| \psi \|_{L^2},
\]
which implies $\lambda_1(V) \ge m - \nrm{V}{L^\infty}$.
\end{proof}

\medskip

The linear lower bound for $\Lambda( α )$ implies in particular that $\alpha_*(p,d) \ge (c_{d,p} m^{\frac{d}{p}-1})^{-1} > 0$.

\section{The Keller problem for the Dirac operator}
\label{sec:Keller}

In what follows, we fix $d \ge 1$ and $p > d$. Recall that the Keller problem $\Lambda(\alpha)$ is defined in~\eqref{eq:def:Lambda_intro} with the constraint $\nrm{V}{L^p} \le \alpha$, so that $\alpha \mapsto \Lambda(\alpha)$ is nonincreasing. 

\begin{lemma} 
    \label{lem:Lambda_first_properties}
    The function $\alpha \mapsto \Lambda(\alpha)$ is nonincreasing and continuous on $(0, \infty)$. In addition,
    \[
        \forall \alpha > 0, \qquad \max \big(0, m - c_{d,p} m^{d/p} \alpha \big) \le \Lambda(\alpha) < m, \qquad \text{and} \qquad \lim_{\alpha \to \infty} \Lambda(\alpha) = 0.
    \]
\end{lemma}

\begin{proof}
The lower bound for $\Lambda(\alpha)$ comes from Lemma~\ref{lem:bounds_lambda1} and the fact that $\Lambda(\alpha) \ge 0$. 

\medskip

For the upper bounds, for $ h > 0$ and $R > 0$, we consider the test potential $V_{h, R} := h \1_{B_R}$, which satisfies $\| V_{h, R} \|_{L^p} = h | B_1 |^{1/p} R^{d/p}$. Using the upper bound of Lemma~\ref{lem:bounds_lambda1}, we have, with $h := α | B_1 |^{-1/p} R^{-d/p}$, 
\[
    \Lambda(\alpha)^2 \le \lambda_1(V_{h,R})^2 \le \lambda_1( - Δ + ( m - V_{h, R})^2).
\]
To give an upper bound for the last eigenvalue, we consider the first Dirichlet eigenfunction of $-\Delta$ on $B_R$ and extend it by zero. The associated eigenvalue is $\ell_R = \ell_1 R^{-2}$. We get
\[
    \Lambda(\alpha)^2 \le \frac{\ell_1}{R^2} + ( m - h)^2 = m^2 - 2 m \frac{α | B_1 |^{-1/p}}{R^{d/p}} + \frac{α^2 | B_1 |^{-2/p}}{R^{2d/p}} + \frac{\ell_1}{R^2}
\]
Since $d/p < 1$, for fixed $\alpha$ and large $R$ we have
\[
 \Lambda(\alpha)^2  \le  m^2 - 2 m \frac{α | B_1 |^{-1/p}}{R^{d/p}}  (1 + o(1)) < m^2,
\]
so $\Lambda(\alpha) < m$. On the other hand, simply taking $h = m$ and $R → \infty$ (so that $α → \infty$ as well) shows that $\Lambda(\alpha) → 0$ as $α → \infty$.

\medskip

It remains to prove the continuity of the map $\Lambda(\alpha)$. We use a scaling argument. For $s > 0$ and $V \in L^p(\R^d, \R)$, we set $V_s(x) := s V(s x)$, which satisfies $\| V_s \|_{L^p} = s^{1 - d/p} \| V \|_{L^p}$. For $\psi \in \Dom(H_V)$, we set $\psi_s(x) := s^{d/2} \psi(s x)$, so that $\| \psi_s \|_{L^2} = \| \psi \|_{L^2}$. A direct computation gives
\[
    \big( (\Dirac_m - V_s \beta) \psi_s \big)(x) = s^{d/2} \Big( s (\Dirac_m - V \beta) \psi + m (1 - s) \beta \psi \Big)(s x).
\]
Since $\beta$ is unitary, we have
\[
    \| (\Dirac_m - V_s \beta) \psi_s \|_{L^2} \le s \| (\Dirac_m - V \beta) \psi \|_{L^2} + m | 1 - s | \, \| \psi \|_{L^2}.
\]
Taking the infimum over all normalized $\psi$ gives $\lambda_1(V_s) \le s \lambda_1(V) + m | 1 - s |$.

\medskip
We then use $\Lambda(s^{1 - d/p} \alpha) \le \lambda_1(V_s)$, optimize the right-hand-side over $V$ with $\| V \|_{L^p} = \alpha$, and use that $\Lambda(\alpha)$ is non-increasing to get, for all $s \ge 1$,
\[
    \Lambda(\alpha) \le \Lambda(s^{1 - d/p} \alpha) \le s \Lambda(\alpha) + m (s - 1).
\]
Setting $\alpha' := s^{1 - d/p} α \ge α$, so that $s = (\alpha'/\alpha)^{\frac{p}{p-d}} \ge 1$, this inequality also reads:
\[
    ∀ α \le α', \quad 0 \le Λ(α') - Λ(α) \le \left( \left( \frac{\alpha'}{\alpha} \right)^{\frac{p}{p-d}} - 1 \right)(Λ(α) + m) \le 2m \left( \left( \frac{\alpha'}{\alpha} \right)^{\frac{p}{p-d}} - 1 \right),
\]
which implies continuity of $α ↦ Λ(α)$.
\end{proof}

\section{Existence of optimal potentials}
\label{sec:existence}

In this section, we prove Theorem~\ref{thm:existence}. We first prove that optimal potentials exist for all $\alpha > 0$. We then derive the Euler--Lagrange equations, and show that optimal potentials saturate the constraint $\nrm{V}{L^p} \le \alpha$ as long as $\Lambda(\alpha) > 0$.

\subsection{A priori bounds and compactness}

\begin{lemma}[A priori bounds on eigenfunctions] 
    \label{lem:regularity_eigenfunctions}
    Let $α  > 0$ and $\lambda_{\max} \in (0, m)$. There are constants $0 < c \le C < \infty$ such that, for all $V$ with $\nrm{V}{L^p} \le α$ and all normalized eigenfunctions $\psi$ of $\Dirac_m - V\beta$ with eigenvalue $\lambda \in [-\lambda_{\max}, \lambda_{\max}]$, the function $\psi$ is continuous and satisfies the bounds
    \begin{align*}
        \nrm{\psi}{W^{1,p}} \le C \quad \text{and} \quad c \le \nrm{\psi}{L^\infty} \le C \quad \text{and} \quad \nrm{\psi}{L^2(B_1(x))} \ge  c \quad \text{for some $x ∈ \R^d$.}
    \end{align*}
\end{lemma}

\begin{proof}
    \noindent \underline{Upper bounds.} From the equation $λ ψ = (\Dirac_m - V \beta) \psi$, we get $\psi = (\Dirac_m - \lambda)^{-1} V \beta \psi$. It is a classical result that $(\Dirac_m - \lambda)^{-1}$ is a bounded operator from $L^s(\R^d)$ to $W^{1,s}(\R^d)$ for all $1 < s < \infty$, with an operator bound independent of $λ ∈ [-\lambda_{\max}, \lambda_{\max}]$ (this comes from Mikhlin multiplier theorem).
    
    Let us first prove that $ψ ∈ L^\infty$. Assume first $p > 2$ (which is always the case if $d \ge 2$). We proceed by induction. We have $ψ ∈ L^{t_0}$ with $t_0 = 2$. If $ψ ∈ L^{t_n}$ for some $t_n \ge 2$, then Hölder's inequality gives $V ψ ∈ L^s$ for all $s ∈ [s_0, s_n]$ with 
    \[
        \frac{1}{s_0} = \frac{1}{p} + \frac{1}{2}, \quad \frac{1}{s_n} = \frac{1}{p} + \frac{1}{t_n},
    \]
    and therefore $\psi \in W^{1,s}$ (note that $s_0 > 1$ since we assumed $p > 2$). In the case where $\frac{1}{t_n} < \frac{1}{d} - \frac{1}{p}$, Morrey's embedding shows that $ψ ∈ L^\infty$ and that $ψ$ is continuous, and we are done. Otherwise, Sobolev's embedding gives $ψ ∈ L^{t_{n+1}}$ with
    \[
        \frac{1}{t_{n+1}} = \frac1s - \frac1d = \frac1{t_n} - \Big( \frac1d - \frac1p \Big) = \frac{1}{t_0} - (n+1) \Big( \frac1d - \frac1p \Big).
    \]
    After finitely many steps, the condition $\frac{1}{t_n} < \frac{1}{d} - \frac{1}{p}$ is satisfied, and we end up with $ψ ∈ L^\infty$ continuous. In addition, since the number of steps is independent of $V$, we obtain $\| ψ \|_{L^\infty} \le C$ independent of $V$ with $\| V \|_{L^p} \le α$. Iterating one more time shows that $ψ ∈ W^{1,p}$ as claimed.
    
    \medskip

    We now focus on the case $1 < p \le 2$, which can only happen if $d = 1$. The eigenvalue equation reads $\psi' = \ri \sigma_1 (λ - (m - V) \sigma_3 ) \psi$, so $| \psi' | \le (2 m + | V |) | ψ |$. On an interval $I$ of length $1$, Grönwall's lemma gives
    \[
        \sup_I | ψ | \le \re^{2m + α} \inf_I | ψ | \le \re^{2m + α} \| ψ \|_{L^2(I)} \le \re^{2m + α}.
    \]
    We deduce that $\| ψ \|_{L^\infty} \le \re^{2m + α}$, and conclude as before that $ψ ∈ W^{1,p}$ and that $\psi$ is continuous.

    \medskip

    \noindent \underline{Lower bound.} 
    Using the eigenvalue equation and Hölder's inequality, we get
    \begin{align*}
    1 & = \| \psi \|_{L^2}^2 = ⟨ ψ, ψ ⟩ = ⟨ ψ, (\Dirac_m - λ)^{-1} V β ψ ⟩ = ⟨ (\Dirac_m - λ)^{-1} ψ, V β ψ ⟩ \\
    & \le \| V \|_{L^p} \| (\Dirac_m - λ)^{-1} ψ \|_{L^{2p'}} \| ψ \|_{L^{2p'}}
    \le C \| V \|_{L^p} \| ψ \|_{L^{2p'}}^2,
    \end{align*}
    where we used that $(\Dirac_m - \lambda)^{-1}$ is a bounded operator from $L^s$ to $L^s$ for all $1 < s < \infty$. The lower bound for the $L^\infty$-norm follows from the interpolation inequality
    \[
        \| ψ \|_{L^{2p'}} \le \| ψ \|_{L^\infty}^{\frac{1}{p}} \| ψ \|_{L^2}^{\frac{1}{p'}} = \| ψ \|_{L^\infty}^{\frac{1}{p}}.
    \]
    Finally, since $\nrm{\psi}{W^{1,p}} \le C$ with $p > d$, Morrey's inequality shows that $\psi$ is of Hölder class $C^{0, \alpha}$ for some $α > 0$, with Hölder constant independent of $V$. So, if $|\psi(x)| \ge \tilde{c}$, then $|\psi| \ge \tilde{c}/2$ on a ball $B_\delta(x)$, with $\delta > 0$ independent of $V$. This gives the lower bound on $\nrm{\psi}{L^2(B_1(x))}$.
\end{proof}

The lower bound prevents eigenfunctions from vanishing. The uniform upper bounds allow us to extract converging subsequences. We obtain the following compactness result.

\begin{lemma}[Compactness] \label{lem:compactness}
    Let $(V_n)$ be a sequence of potentials with $\nrm{V_n}{L^p} \le α$, and let $\lambda_n$ be an eigenvalue of $\Dirac_m - V_n \beta$, with $|\lambda_n| \le \lambda_{\max} < m$ and $\lambda_n \to \lambda$. Then, up to a subsequence and up to translations, $(V_n)$ converges weakly in $L^p$ to some $V$ such that $\lambda$ is an eigenvalue of $\Dirac_m - V\beta$. In particular, 
    \[
        \nrm{V}{L^p} \le \liminf_{n \to \infty} \nrm{V_n}{L^p} \le α \qquad \text{and} \qquad \lambda_1(V) \le |\lambda|.
    \]
\end{lemma}

\begin{proof}
    Let $\psi_n$ be a normalized eigenfunction of $\Dirac_m - V_n\beta$ with eigenvalue $\lambda_n$. Lemma~\ref{lem:regularity_eigenfunctions} shows that there is $x_n \in \R^d$ with $\nrm{\psi_n}{L^2(B_1(x_n))} \ge c$. Up to translations of $V_n$, we may assume that $x_n = 0$, that is $\nrm{\psi_n}{L^2(B_1)} \ge c$. 
    
    The sequence $V_n$ is bounded in $L^p$ and the sequence $\psi_n$ is bounded in $W^{1,p}$. Applying the Banach--Alaoglu theorem, we have, up to a subsequence,
    \[
        V_n \rightharpoonup V \quad \text{weakly in } L^p \qquad \text{and} \qquad \psi_n \rightharpoonup \psi \quad \text{weakly in } W^{1,p}.
    \]
    In particular, for any compact set $K$, we have $\psi_n → \psi$ uniformly on $K$, by the compact embedding $W^{1,p}(K) \hookrightarrow C^0(K)$ (recall that $p > d$).

    Consider a test function $χ ∈ C^\infty_0(\R^d, \C^N)$, and let $K$ be a compact set containing the support of $\chi$. Then
    \[
        0 = ⟨ \chi, (\Dirac_m - V_n β - \lambda_n)  \psi_n ⟩ = ⟨ \Dirac_m \chi,  \psi_n ⟩ - \int_{K} V_n ⟨ β \chi, \psi_n ⟩_{\C^N} - \lambda_n ⟨ \chi, \psi_n ⟩.
    \]
    The first term converges to $⟨ \Dirac_m \chi, ψ ⟩$ by the weak convergence of $\psi_n$ to $\psi$ in $L^2$. Similarly, the third term converges to $\lambda ⟨ \chi, \psi ⟩$. For the middle term, we use that $V_n \rightharpoonup V$ weakly in $L^p$ while $⟨ β \chi, \psi_n ⟩_{\C^N} → ⟨ β \chi, \psi ⟩_{\C^N}$ strongly in $L^{p'}(K)$. Taking the limit shows that
    \[
    0 = ⟨ (\Dirac_m - V β - \lambda) \chi, ψ ⟩ = ⟨ \chi, (\Dirac_m - V β - \lambda) ψ ⟩.
    \]
    Since this holds for all $χ ∈ C^\infty_0(\R^d, \C^N)$, we deduce that $(\Dirac_m - V β - \lambda) ψ = 0$ in the distributional sense. On the other hand, we have $ψ ∈ L^2$ and $\nrm{\psi}{L^2(B_1)} \ge c$, by the strong convergence of $\psi_n$ to $\psi$ in $L^2(B_1)$. So $\psi$ is an eigenfunction of $\Dirac_m - V β$, hence $λ ∈ σ (\Dirac_m - V \beta)$. Since $\lambda_1(V)$ is the smallest eigenvalue in magnitude, we get as wanted $\lambda_1(V) \le | λ |$.
\end{proof}

As a consequence, we deduce the existence of optimizers for our Keller problem. 

\begin{proposition}[Existence of optimal potentials] \label{prop:existence}
    For all $\alpha > 0$, there is $V \in L^p(\R^d, \R)$ with $\nrm{V}{L^p} \le \alpha$ and $\lambda_1(V) = \Lambda(\alpha)$. In addition, $\Lambda(\alpha)$ is an eigenvalue of $\Dirac_m - V\beta$.
\end{proposition}

\begin{proof}
    Let $(V_n)$ be a minimizing sequence, with $\nrm{V_n}{L^p} \le \alpha$ and $\lambda_1(V_n) \to \Lambda(\alpha)$. Since $\Lambda(\alpha) < m$ by Lemma~\ref{lem:Lambda_first_properties}, we may assume that $\lambda_1(V_n) \le \lambda_{\max} < m$ for all $n$. Applying Lemma~\ref{lem:compactness} provides $V$ with $\nrm{V}{L^p} \le \alpha$ and $\lambda_1(V) \le \Lambda(\alpha)$. By definition of $\Lambda(\alpha)$, we have $\Lambda(\alpha) \le \lambda_1(V)$, hence the equality $\lambda_1(V) = \Lambda(\alpha)$.
\end{proof}

\subsection{The Euler--Lagrange equation}

We now discuss the properties of the optimizers $V$. Let us first write a {\em virial identity}.

\begin{proposition}[Virial identity]
    \label{prop:virial}
    Let $\alpha > 0$ be such that $\Lambda(\alpha) > 0$, $V$ be an optimal potential for $\Lambda(\alpha)$, and $\psi$ be a normalized eigenfunction of $\Dirac_m - V\beta$ with eigenvalue $\Lambda(\alpha)$. Then,
    \[
        ⟨ \psi, \Dirac_0 \psi ⟩  = \frac{d}{p} \int_{\R^d} V ⟨ \psi, β \psi ⟩_{\C^N}.
    \]
\end{proposition}

\begin{proof}
    For $s > 0$, we set $V_s(x) := s^{\frac{d}{p}} V(sx)$, so that $\| V_s \|_{L^p} = \| V \|_{L^p} \le \alpha$, hence $\lambda_1(V_s) \ge \Lambda(\alpha)$. We also introduce the scaling unitary operator $(U_s \psi)(x) := s^{d/2} \psi(sx)$. A computation shows that
    \[
        U_s^* (\Dirac_m - V_s \beta) U_s = s \Dirac_0 + m \beta - s^{d/p} V \beta.
    \]
    With $\psi_s := U_s \psi$ as a test function, writing $s = 1 + \eta$ and $\Lambda := \Lambda(\alpha)$, and using that $(1 + \eta)^{d/p} = 1 + \frac{d}{p} \eta + O(\eta^2)$, we obtain
    \begin{align*}
        Λ^2 & \le \| ((\Dirac_m - V_s \beta)) \psi_s \|_{L^2}^2
        = \| Λ ψ + (\eta \Dirac_0 - [(1 + \eta)^{d/p} - 1] V β) ψ \|_{L^2}^2  \\
        & \le \Lambda^2 + 2 Λ \eta \left( ⟨ \psi, \Dirac_0 ψ ⟩ - \frac{d}{p}⟨ \psi, V β ψ ⟩ \right)  + O(\eta^2).
    \end{align*}
    Note that the term in parenthesis is real-valued. The claim follows since this inequality must hold for all $\eta$ small enough, with any sign. 
\end{proof}

We can now derive the Euler--Lagrange equations.

\begin{proposition}[Euler--Lagrange equations] \label{prop:EL}
    Let $\alpha > 0$ be such that $\Lambda(\alpha) > 0$, let $V$ be an optimal potential for $\Lambda(\alpha)$, and let $\psi$ be a normalized eigenfunction of $\Dirac_m - V\beta$ with eigenvalue $\Lambda(\alpha)$. Then $\sprod{\psi}{\beta\psi} \not\equiv 0$, we have $\nrm{V}{L^p} = \alpha$, and there is $\mu > 0$ such that
    \[
        \mu \, |V|^{p-2}V = \sprod{\psi}{\beta\psi}.
    \]
    In particular, $ψ ∈ C^1(\R^d, \C^N)$ and $V ∈ C^0(\R^d, \R)$ is bounded.
\end{proposition}

\begin{proof}
    We write $λ := \Lambda(\alpha) = \lambda_1(V)$ for simplicity. Recall that we assume $λ > 0$.

    \noindent \textbf{Step 1: first-order condition.} Let $W \in L^\infty \cap L^p(\R^d, \R)$. Using $\psi$ as a test function, we have
    \begin{equation} \label{eq:first_order}
        \lambda_1(V + tW)^2 \le \nrm{(H_V - t W \beta) \psi}{L^2}^2 = \lambda^2 - 2 t \lambda \int_{\R^d} W \sprod{\psi}{\beta\psi} + t^2 \nrm{W \psi}{L^2}^2.
    \end{equation}
    Assume first that $\nrm{V}{L^p} < \alpha$. Then $\nrm{V + tW}{L^p} \le \alpha$ for all $t$ small enough, and $\lambda_1(V + tW) \ge \lambda$. We deduce that $\int W ⟨ \psi, \beta\psi⟩ = 0$ for all $W$, that is $\sprod{\psi}{\beta\psi} \equiv 0$. The virial identity Proposition~\ref{prop:virial} then shows that $⟨ \psi, \Dirac_0 ψ ⟩ = 0$. As a consequence,
    \[
        λ = ⟨ \psi, (\Dirac_m - V \beta) ψ ⟩ = ⟨ \psi, \Dirac_0 ψ ⟩ + m ⟨ \psi, β ψ ⟩ - ⟨ \psi, V β ψ ⟩ = 0,
    \]
    which is a contradiction, since we assumed $λ > 0$. So $\| V \|_{L^p} = \alpha$, as claimed. The same argument shows that $⟨ \psi, β ψ ⟩ \not\equiv 0$.

    \medskip
    
    The potential $V$ is an optimizer for the following minimization problem, where we fix the test function $\psi$:
    \[
        \inf \left\{ \| ( \Dirac_m - W β ) ψ \|_{L^2}^2, \quad \| W \|_{L^p}^p \le \alpha^p \right\}.
    \]
    Since $V$ saturates the constraints, it satisfies the Euler--Lagrange equations
    \[
        ⟨ ψ, β ψ ⟩_{\C^N} = μ | V |^{p-2} V, \qquad \text{for some} \quad μ > 0.
    \]
    The sign of $V$ is therefore pointwise the one of $⟨ ψ, β ψ ⟩_{\C^N}$. One can invert this equation and write
    \begin{equation} \label{eq:EL_bis}
        V = \mu^{\frac{-1}{p-1}} | ⟨ ψ, β ψ ⟩_{\C^N} |^{\frac{1}{p-1} - 1} ⟨ ψ, β ψ ⟩_{\C^N}.
    \end{equation}
    Since $\psi$ is continuous and bounded (see Lemma~\ref{lem:regularity_eigenfunctions}), so is $V$. Bootstrapping as in the proof of Lemma~\ref{lem:regularity_eigenfunctions} shows that $\psi$ is differentiable. Actually, since the function $s \mapsto |s|^{\frac{1}{p-1} - 1} s$ is smooth except at $s = 0$, $\psi$ and $V$ are smooth on the open set $\{ ⟨ ψ, β ψ ⟩_{\C^N} \neq 0 \}$.
\end{proof}

Taking the $p$ power of~\eqref{eq:EL_bis} and integrating shows that, with $p' = \frac{p}{p-1}$,
\[
    \mu^{p'} \alpha^p = \| ⟨ ψ, β ψ ⟩_{\C^N} \|_{L^{p'}}^{p'}, \quad \text{hence} \quad \mu = \alpha^{1-p} \, \| ⟨ ψ, β ψ ⟩_{\C^N} \|_{L^{p'}}.
\]

\begin{remark}
    From the virial identity and the Euler--Lagrange equations, we get
    \begin{equation} \label{eq:Pohozaev}
    \Lambda(\alpha) = ⟨ \psi, \Dirac_0 ψ ⟩ + m \int_{\R^d} ⟨ \psi, β ψ ⟩_{\C^N} - \int_{\R^d} V  ⟨ ψ, β ψ ⟩_{\C^N}
    = - \left(1 - \frac{d}{p} \right) μ \alpha^p + m \int_{\R^d}⟨ \psi, β ψ ⟩_{\C^N}.
    \end{equation}
    In particular, $\int_{\R^d}⟨ \psi, β ψ ⟩_{\C^N} > m^{-1} \Lambda(\alpha) > 0$.
\end{remark}

We were not able to prove that $V \ge 0$ pointwise.  This holds in dimension $d = 1$ (see Appendix~\ref{sec:appendix:explicit_1D}), but our simulations in dimension $d = 2$ suggest that it fails for large $\alpha$ (see Appendix~\ref{sec:appendix:numerics}). We conjecture that the optimal $V$ is positive for sufficiently small $\alpha$. This would follow from an a-priori bound $\nrm{V}{L^\infty} < m$ such that $m - V > 0$, in which case the lowest positive eigenvalue can be found by a variational problem for the upper component. 

\medskip

We now provide the proof of Theorem~\ref{thm:existence}.

\begin{proof}[Proof of Theorem~\ref{thm:existence}]
The properties of $\Lambda$ on $(0, \infty)$ are given in Lemma~\ref{lem:Lambda_first_properties}. The existence of optimal potentials for all $\alpha > 0$ is Proposition~\ref{prop:existence}. For $\alpha < \alpha_*$, Proposition~\ref{prop:EL} shows that optimal potentials are bounded and saturate the constraint, and gives the Euler--Lagrange equations.

\medskip

It remains to prove that $\Lambda$ is (strictly) decreasing on $(0, \alpha_*)$. Let $0 < \alpha < \alpha' < \alpha_*$, and assume that $\Lambda(\alpha) = \Lambda(\alpha')$. An optimal potential $V$ for $\alpha$ satisfies
\[
    \nrm{V}{L^p} \le \alpha < \alpha' \qquad \text{and} \qquad \lambda_1(V) = \Lambda(\alpha) = \Lambda(\alpha').
\]
So $V$ is also an optimal potential for $\alpha'$, which does not saturate the constraint. Since $\Lambda(\alpha') > 0$, this contradicts Proposition~\ref{prop:EL}.
\end{proof}

We conclude this section with the following characterization of the critical mass.

\begin{lemma}[The critical mass is attained] \label{lem:alpha_star}
    If $\alpha_*(p,d) < \infty$, there is $V \in L^p(\R^d, \R)$ with $\nrm{V}{L^p} = \alpha_*(p,d)$ such that $0$ is an eigenvalue of $\Dirac_m - V\beta$.
    Conversely, $\alpha_*(p, d) = + \infty$ if and only if $\Dirac_m - V\beta$ has no zero mode for all $V \in L^p(\R^d, \R)$.
\end{lemma}

\begin{proof}
    If $\alpha_* < \infty$, we have $\Lambda(\alpha_*) = 0$ by continuity. Proposition~\ref{prop:existence} provides $V$ with $\nrm{V}{L^p} \le \alpha_*$ such that $0 = \Lambda(\alpha_*)$ is an eigenvalue of $\Dirac_m - V\beta$. Then $\Lambda(\nrm{V}{L^p}) = 0$, so that $\nrm{V}{L^p} \ge \alpha_*$ by definition of $\alpha_*$. Conversely, if $\Dirac_m - V\beta$ has a zero mode, then $\Lambda(\nrm{V}{L^p}) = 0$, so $\alpha_* \le \nrm{V}{L^p}$.
\end{proof}

\section{No zero modes in dimensions \texorpdfstring{$1$}{1} and \texorpdfstring{$2$}{2}}
\label{sec:no_zero_modes}

In this section, we prove Theorem~\ref{thm:critical_alpha_dim12} that there are no zero modes in dimensions $1$ and $2$. By Lemma~\ref{lem:alpha_star}, it is enough to prove the following.

\begin{theorem} \label{thm:no_zero_modes}
    Let $d \in \{1, 2\}$, $p > d$ and $V \in L^p(\R^d, \R)$. Then $0$ is not an eigenvalue of $\Dirac_m - V\beta$.
\end{theorem}

This eventually proves that $\alpha_*(p,d) = +\infty$ by Lemma~\ref{lem:alpha_star}. In what follows, we set $M := m - V$.

\subsection{The case \texorpdfstring{$d = 1$}{d=1}}
We choose a spin basis in which $α_1 = \sigma_1$ and $\beta = \sigma_3$. The zero mode equation $(\Dirac_m - V β) ψ = 0$ is $- \ri \sigma_1 \psi' + M \sigma_3 ψ = 0$. Multiplying by $\ri \sigma_1$, this is also
\[
    \psi' = - M \sigma_2 \psi = \begin{pmatrix} 0 & \ri M \\ - \ri M & 0 \end{pmatrix} \psi.
\]
This is a simple ODE, whose solution is of the form
\[
    \psi(x) = c_+ \exp \left( - \int_0^x M\right) \begin{pmatrix} 1 \\ \ri \end{pmatrix} + c_- \exp \left( \int_0^x M\right) \begin{pmatrix} 1 \\ - \ri \end{pmatrix}
\]
for some $c_\pm \in \C$. Since $V ∈ L^p(\R, \R)$, Hölder's inequality gives $| \int_0^x V | \le \nrm{V}{L^p} |x|^{1 - 1/p} = o(| x |)$, so $\int_0^x M = mx + o(| x |)$ as $x \to \pm \infty$. The integrability condition $\psi \in L^2(\R, \C^2)$ forces $c_\pm = 0$, hence $\psi = 0$.

\subsection{The case \texorpdfstring{$d = 2$}{d=2}}

Recall that we have defined $\alpha_1 = \sigma_1$, $\alpha_2 = \sigma_2$ and $\beta = \sigma_3$. Let $ψ ∈ L^2(\R^2, \C^2)$ be a solution to the eigenvalue equation $(\Dirac_m - V \beta) ψ  = 0$. By Lemma~\ref{lem:regularity_eigenfunctions}, $\psi$ is continuous and in $W^{1,p}(\R^2) \cap L^\infty(\R^2)$.

The equation $(\Dirac_m - V) ψ = (\Dirac_0 + M \beta) ψ  = 0$ reads
\[
    \begin{cases} 
        2 \partial_z \psi_2 = M (- \ri \psi_1) \\
        2 \partial_{\overline z} (- \ri \psi_1) = M \psi_2
    \end{cases}, \qquad \text{with} \qquad
    \partial_z = \frac12 (\partial_1 - \ri \partial_2), 
    \quad 
    \partial_{\overline{z}} = \frac12 (\partial_1 + \ri \partial_2).
\]
We conjugate the second equation (recall that $M$ is real-valued), add and subtract the two equations, and get that the functions $u := \psi_2 + \ri \overline{\psi_1}$ and $v := \psi_2 - \ri \overline{\psi_1}$ solve
\[
    2 \partial_z u = M \overline{u}, \quad \text{and} \quad 2 \partial_z v = - M \overline{v}.
\]
Let us prove that this equation implies $u = 0$ (the proof for $v$ is identical, with $M$ replaced by $-M$). The function $F := u^2$ solves
\[
    \partial_z F = 2 u \partial_z u = M | u |^2 = M | F |.
\]
The right-hand side is real-valued. Writing $F = f + \ri g$ with $f$ and $g$ real-valued, the equation becomes
\[
\partial_2 f - \partial_1 g = 0, \qquad \partial_1 f + \partial_2 g = 2 M | F |.
\]
The first equation shows that the continuous form $f \rd x_1 + g \rd x_2$ is closed, hence exact on $\R^2$: there is a real-valued function $\Theta \in C^1(\R^2, \R)$ such that $\nabla \Theta = (f, g)$. Since $|F| = (f^2 + g^2)^{1/2} = |\nabla \Theta|$, the second equation becomes
\begin{equation} \label{eq:Theta}
    \Delta \Theta = 2 M | \nabla \Theta |.
\end{equation}
In addition, we have $|\nabla \Theta| = |u|^2 \in L^1(\R^2) \cap L^\infty(\R^2)$. It remains to prove that~\eqref{eq:Theta} has no nontrivial solution.

\medskip

We first use an application of a Sobolev-Poincaré inequality, see e.g. ~\cite[Section 7.8]{GilTru-01}.

\begin{lemma}\label{lem:poincare_sobolev}
    If $f ∈ L^1_{\rm loc}(\R^2, \R)$ is such that $∇ f ∈ L^1(\R^2, \R^2)$, then there is $f_∞ ∈ \R$ such that $f - f_∞ ∈ L^2(\R^2)$.
\end{lemma}

\begin{proof}
    We claim that there is a constant $C > 0$ so that, for all $R > 0$ and all $f ∈ W^{1,1}(B_R)$, we have 
\[
    \nrm{f - \overline{f}_R}{L^2(B_R)} \le C \nrm{\nabla f}{L^1(B_R)}\qquad \text{with} \qquad \overline{f}_R := \frac{1}{|B_R|} \int_{B_R} f.
\]
By scaling, the constant $C$ is independent of $R$, so it is enough to prove the inequality for one fixed $R$. Assume otherwise, and let $f_n \in W^{1,1}(B_R)$ be a sequence with
\[
    \overline{(f_n)}_R = 0, \qquad \nrm{f_n}{L^2(B_R)} = 1, \qquad \nrm{\nabla f_n}{L^1(B_R)} \to 0.
\]
By Hölder's inequality, $\nrm{f_n}{L^1(B_R)} \le |B_R|^{1/2} \nrm{f_n}{L^2(B_R)}$, so $(f_n)$ is bounded in $W^{1,1}(B_R)$. By Rellich--Kondrachov's theorem, up to a subsequence, $f_n \to f$ strongly in $L^1(B_R)$. The limit satisfies $\nabla f = 0$ in the distributional sense, so $f$ is constant on the connected set $B_R$, and $f = 0$ since $\overline{f}_R = 0$. Finally, the Sobolev embedding $W^{1,1}(B_R) \hookrightarrow L^2(B_R)$ gives
\[
    1 = \nrm{f_n}{L^2(B_R)} \le C \left( \nrm{f_n}{L^1(B_R)} + \nrm{\nabla f_n}{L^1(B_R)} \right) \xrightarrow[n \to \infty]{} 0,
\]
a contradiction. Next, we apply this inequality on $B_{2R}$. Hölder's inequality gives
\[
    | \overline{f}_R - \overline{f}_{2R} | \le \frac{1}{| B_R |} \int_{B_R} | f(x) - \overline{f}_{2R}| \rd x \le  \frac{1}{| B_R |^{1/2}} \nrm{f - \overline{f}_{2R}}{L^2(B_{2R})} \le \frac{C}{\sqrt{\pi} R} \nrm{\nabla f}{L^1(\R^2)},
\]
so $(\overline{f}_{2^j})_{j \in \N}$ is a Cauchy sequence, which converges to some $f_\infty \in \R$. Then $\nrm{f - f_\infty}{L^2(B_{2^j})}$ is bounded uniformly in $j$, and Fatou's lemma gives $f - f_∞ ∈ L^2(\R^2)$, as wanted.
\end{proof}

Now, we return to the system for $\Theta$.
\begin{lemma}
    If $Θ ∈ C^1(\R^2, \R)$ solves \eqref{eq:Theta} in the distributional sense, with $∇ Θ ∈ L^1(\R^2, \R^2)$, then $\Theta$ is constant.
\end{lemma}

\begin{proof}

Lemma~\ref{lem:poincare_sobolev} shows that there exists $\Theta_\infty$ such that $\Theta- \Theta_\infty \in L^2$. 
Since $\Theta$ is in addition Lipschitz ($\nabla \Theta \in L^\infty$), this implies that $\Theta(x) \to \Theta_\infty$ as $|x| \to \infty$. In what follows, we assume without loss of generality $\Theta_∞ = 0$. It remains to prove that~\eqref{eq:Theta} has no nontrivial solution going to $0$ at infinity. We use a maximum principle, following the truncation method of Stampacchia~\cite{Sta-65} (see also~\cite[Chapter~8]{GilTru-01}). Assume by contradiction that $k := \sup \Theta > 0$. For $\ell \in (0, k)$, we set
\[
    \phi_\ell := [\Theta - \ell]_+, \qquad S_\ell := \{ \Theta > \ell \}, \qquad T_\ell := \{ \Theta > \ell \} \cap \{ \nabla \Theta \neq 0 \}.
\]
Since $\Theta$ is continuous and goes to $0$ at infinity, the set $S_\ell$ is open and bounded. We multiply~\eqref{eq:Theta} by $\phi_\ell$ and integrate. Hölder's inequality gives, with $q := \frac{2p}{p-2} \in (2, \infty)$ (recall that $p > d = 2$)
\[
    \int_{\R^2} |\nabla \phi_\ell|^2 = - 2 \int_{T_\ell} M \, \phi_\ell \, |\nabla \phi_\ell|
    \le 2 \nrm{M}{L^p(T_\ell)} \nrm{\phi_\ell}{L^q(T_\ell)} \nrm{\nabla \phi_\ell}{L^2(T_\ell)},
\]
so
\begin{equation} \label{eq:ineq:Sta1}
    \nrm{\nabla \phi_\ell}{L^2(T_\ell)} \le 2 \nrm{M}{L^p(T_\ell)} \nrm{\phi_\ell}{L^q(T_\ell)}.
\end{equation}
On the other hand, the Gagliardo--Nirenberg inequality in dimension $2$, together with Hölder's inequality on the support $S_\ell$ of $\phi_\ell$, gives
\[
    \nrm{\phi_\ell}{L^q(S_\ell)} \le C_q \nrm{\phi_\ell}{L^2(S_\ell)}^{\frac{2}{q}} \nrm{\nabla \phi_\ell}{L^2(S_\ell)}^{1 - \frac{2}{q}}
    \quad \text{and} \quad
    \| \phi_ℓ \|_{L^2(S_\ell)} \le |S_\ell|^{\frac12 - \frac1q} \nrm{\phi_\ell}{L^q(S_\ell)}.
\]
Combining the two inequalities and using~\eqref{eq:ineq:Sta1} gives
\[
    \nrm{\phi_\ell}{L^q(S_\ell)} \le C_q' \, |S_\ell|^{\frac1q} \nrm{\nabla \phi_\ell}{L^2(T_\ell)} \le C_q' \, |S_\ell|^{\frac1q} 2 \nrm{M}{L^p(T_\ell)} \nrm{\phi_\ell}{L^q(T_\ell)}.
\]
Since $\phi_\ell \neq 0$, we have $\nrm{\phi_\ell}{L^q(S_\ell)} \neq 0$, and we obtain
\[
    1 \le 2 C_q' \, \nrm{M}{L^p(T_\ell)} \, |S_\ell|^{\frac1q} \qquad \text{for all } \ell \in (0, k).
\]
We now let $\ell \to k$. The sets $T_\ell$ are nonincreasing, and their intersection is $\{ \Theta = k \} \cap \{ \nabla \Theta \neq 0 \}$, which is the empty set, since every point where $\Theta = k$ is a global maximum of the $C^1$ function $\Theta$. So $|T_\ell| \to 0$, and $\nrm{M}{L^p(T_\ell)} \to 0$ as well. So we must have $| S_ℓ | → \infty$ as $ℓ → k$, which contradicts $S_\ell \subset S_{\ell_0}$ for $\ell \ge \ell_0$, with $S_{\ell_0}$ bounded. This proves that $k = 0$, and therefore that $Θ \le 0$. 

\medskip

The exact same argument with $- \Theta$ (and replacing $M$ by $-M$) shows that $Θ \ge 0$, and we conclude that $Θ = 0$. We deduce that $F = u^2 = 0$, and $u = 0$. The proof for $v$ is identical, with $M$ replaced by $-M$. So $\psi = 0$, and $\Dirac_m - V\beta$ has no zero mode in dimension $d = 2$.
\end{proof}

\subsection{Non existence of zero modes in the radial case}
We now prove Proposition~\ref{prop:radialV}, and prove that if $V ∈ L^p(\R^d, \R)$ is radial, then there are no zero modes. We already proved the absence of zero mode in dimensions $1$ and $2$, so we may assume $d \ge 3$. 

\medskip

Consider $V ∈ L^p(\R^d, \R)$ radial, and write $M(r) := m - V(r)$, where we set $x = r \omega$ with $r = | x |$ and $ω ∈ \mathbb{S}^{d-1}$. We first write the Dirac operator in polar coordinates. From the Clifford relations $\alpha_j \alpha_k + \alpha_k \alpha_j = 2 \delta_{jk}$, we have the formal computation (we do not assume that the $x_j$'s and the $y_k$'s commute)
\[
    (\balpha ⋅ \bx)(\balpha ⋅ \by) = \sum_{1 \le j, k \le d} \alpha_j \alpha_k x_j y_k = \sum_{1 \le j \le d} x_j y_j  + \sum_{ 1 \le j< k \le d} \alpha_j \alpha_k (x_j y_k - x_k y_j).
\]
With $L_{jk} := x_j \partial_k - x_k \partial_j$ and $\widetilde{K} := \sum_{1 \le j < k \le d} \alpha_j \alpha_k L_{jk}$, we obtain in particular
\begin{equation} \label{eq:polar_products}
    (\balpha ⋅ \bx)^2 = | \bx |^2, \qquad
    (\balpha ⋅ \bx)(\balpha ⋅ \nabla) = r \partial_r + \widetilde{K}, \qquad
    (\balpha ⋅ \nabla)(\balpha ⋅ \bx) = d + r \partial_r - \widetilde{K}.
\end{equation}
We set
\[
    S := \balpha \cdot \omega \qquad \text{and} \qquad K := \frac{d-1}{2} - \widetilde{K}.
\]
The operator $S$ is the multiplication operator satisfying $S^2 = \bbI_N$, and $K$ is the {\em spin--orbit operator}. The constant $\frac{d-1}{2}$ is chosen so that $K$ anticommutes with $S$, as we recall in the next Lemma.

\begin{lemma}[Polar decomposition] \label{lem:polar}
    Let $d \ge 2$. We have
    \begin{enumerate}[label=(\roman*)]
        \item $\balpha \cdot \nabla = S \left( \partial_r + \dfrac{d-1}{2r} - \dfrac{1}{r} K \right)$;
        \item $SK = - KS$ and $\beta K = K \beta$;
        \item $\left( K - \frac12 \right)^2 = - \Delta_{\mathbb{S}^{d-1}} + \frac{(d-2)^2}{4}$. In particular, $K$ is self-adjoint with compact resolvent and, if $d \ge 3$, $K$ is injective.
    \end{enumerate}
\end{lemma}

\begin{proof}
    (i) The operator $S$ is a multiplication operator which depends only on $\omega$, so it commutes with $r \partial_r$. Using that $r = | \bx |$, we can check that $\widetilde{K}$ commutes with $r$. Since $\balpha ⋅ \bx = r S$ and $S^2 = \bbI_N$, the second identity of~\eqref{eq:polar_products} gives
    \[
        \balpha ⋅ \nabla = \frac{1}{r} S (r \partial_r + \widetilde{K}) = S \left(\partial_r + \frac{d-1}{2r} - \frac{1}{r} K \right).
    \]
    Using the third identity instead, and the fact that $(r \partial_r) \frac{1}{r} = \partial_r - \frac{1}{r}$, we get
    \[
    \balpha ⋅ \nabla = \left( d + r \partial_r - \widetilde{K} \right) \frac{1}{r} S = \left( \partial_r + \frac{d-1}{r} - \frac{\widetilde{K}}{r} \right) S = \left( \partial_r + \frac{d-1}{2r} + \frac{1}{r} K \right) S.
    \]

    (ii) Comparing the two expressions, and using that $S$ commutes with $\partial_r$ and $\frac{1}{r}$, gives directly $KS + SK = 0$. On the other hand, since $\beta$ anticommutes with all $\alpha_j$, it commutes with $K$.

    \medskip

    (iii) We set $D_\pm := \partial_r + \frac{d-1}{2r} \pm \frac1r K$. By (i) and (ii), we have $S D_- S = D_+$, so
    \[
        \Delta = (\balpha \cdot \nabla)^2 = S D_- S D_- = D_+ D_- = \partial_r^2 + \frac{d-1}{r} \partial_r + \frac{1}{r^2} \left( \frac{(d-1)(d-3)}{4} + K - K^2 \right).
    \]
    Comparing with the usual $\Delta = \partial_r^2 + \frac{d-1}{r} \partial_r + \frac{1}{r^2} \Delta_{\mathbb{S}^{d-1}}$, we obtain 
    \[
        K^2 - K = - \Delta_{\mathbb{S}^{d-1}} + \frac{(d-1)(d-3)}{4},
    \]
    which is the identity. Since $- \Delta_{\mathbb{S}^{d-1}}$ has compact resolvent, so does $K$. If  $K f = 0$, we get $\int_{\mathbb{S}^{d-1}} |\nabla_{\mathbb{S}} f|^2 + \frac{(d-1)(d-3)}{4} |f|^2 = 0$, so for $d>3$ we directly see that $f=0$. For $d=3$, observe that $f$ is constant, but then $\widetilde{K} f = 0$, so $0 = K f = \frac{d-1}{2} f$, and $f = 0$ in this case as well.
\end{proof}

\begin{proof}[Proof of Proposition~\ref{prop:radialV}]
    Let $\psi \in H^1(\R^d, \C^N)$ with $(\Dirac_m - V \beta) \psi = 0$. Multiplying by $\beta$, the equation reads
    \begin{equation} \label{eq:zero_mode_gamma}
        (\boldsymbol{\gamma} \cdot \nabla) \psi = M(r) \psi, \qquad \text{with} \qquad \gamma_j := \ri \beta \alpha_j.
    \end{equation}
    The matrices $\gamma_j$ are hermitian and satisfy the Clifford relations $\gamma_j \gamma_k + \gamma_k \gamma_j = 2 \delta_{jk}$. In addition, $\gamma_j \gamma_k = \alpha_j \alpha_k$ for $j \neq k$, so the spin--orbit operator is unchanged, that is $K_γ = K$. Since the proof of Lemma~\ref{lem:polar} only uses the Clifford relations, it holds with $\balpha$ and $S$ replaced by $\boldsymbol{\gamma}$ and $S_\gamma := \boldsymbol{\gamma} \cdot \omega$ respectively. 
    
    \medskip

    We set $u(r)(\cdot) := r^{\frac{d-1}{2}} \psi(r \, \cdot)$, so that $u$ satisfies the equation
    \begin{equation} \label{eq:radial_u}
        u' = \frac{1}{r} K u + M(r) S_\gamma u, \quad \text{in} \quad L^2(\bbS^{d-1}, \C^N).
    \end{equation}

    \medskip

    Since $S_γ K = - K S_\gamma$ and since $K$ has compact resolvent and is injective, its spectrum is discrete and symmetric with respect to $0$. Fix $\kappa > 0$ an eigenvalue of $K$, and let $P_{\pm}$ be the orthogonal projectors onto $\operatorname{Ker}(K \mp \kappa)$. Since $S_γ K = - K S_\gamma$, we have $S_γ P_\pm = P_\mp S_\gamma$. On the other hand, since $M(r)$ is radial, we have $K M(r) = M(r) K$, so $M(r)$ commutes with $P_\pm$. We set $u_\pm := P_\pm u$. Projecting~\eqref{eq:radial_u} on $P_\pm$ leads to
    \[
        u_+' = \frac{\kappa}{r} u_+ + M(r) \, S_\gamma u_-, \qquad u_-' = - \frac{\kappa}{r} u_- + M(r) \, S_\gamma u_+.
    \]
    In dimension $d = 3$, with $\psi = f(r) \Omega_{\kappa, m} + g(r) \Omega_{-\kappa, m}$ in the basis of spinor spherical harmonics, this is the usual radial system~\cite{Tha-92}.

    \medskip

    We now set $F(r) := \| u_+(r) \|^2_{L^2(\mathbb{S}^{d-1})}$ and $G(r) := \| u_-(r) \|^2_{L^2(\mathbb{S}^{d-1})}$. Differentiating gives
    \[
        \begin{cases}
            F'(r) & = 2 \dfrac{\kappa}{r} F + 2 M(r) {\rm Re} \, ⟨ u_+, S_γ u_- ⟩_{L^2(\bbS^{d-1})} \\
            G'(r) & = - 2 \dfrac{\kappa}{r} G + 2 M(r) {\rm Re} \, ⟨ u_-, S_γ u_+ ⟩_{L^2(\bbS^{d-1})}
        \end{cases}
    \]
    Since $S_\gamma$ is hermitian, we have ${\rm Re} \,  ⟨ u_+, S_\gamma u_- ⟩ = {\rm Re} \, ⟨ u_-, S_\gamma u_+ ⟩$. Subtracting, we get that $W := F - G$ solves
    \[
        W' = \frac{2 \kappa}{r} \left( F + G \right) \ge 0.
    \]
    So $W$ is nondecreasing. Recall that since $\psi$ is an eigenfunction, it is bounded by Lemma~\ref{lem:regularity_eigenfunctions}, and we get 
    \[
        |W(r)| \le \| u(r) \|^2 \le | \mathbb{S}^{d-1} |  r^{d-1} \| \psi \|_{L^\infty}^2.
    \]
    Taking the limit $r → 0$ shows that $W(0^+) = 0$, hence $W \ge 0$. On the other hand, we have
    \[
        \int_0^\infty |W| \rd r \le \int_0^\infty \| u(r) \|^2 \rd r = \| \psi \|_{L^2}^2 < \infty.
    \]
    We deduce that $W = 0$, hence $W' = 0$, and finally $F = G = 0$. Since this holds for all eigenvalues $\kappa > 0$, we deduce that $u = 0$, hence $\psi = 0$.
\end{proof}

\appendix
\addtocontents{toc}{\protect\setcounter{tocdepth}{1}}  

\section{Explicit solutions in dimension \texorpdfstring{$d=1$}{d=1}}
\label{sec:appendix:explicit_1D}

In dimension $d = 1$, the optimal potentials can be computed explicitly. We use $\balpha = \sigma_1$, $\beta = \sigma_3$, and set
\[
    q := \frac{1}{p-1}, \qquad \kappa := \sqrt{m^2 - \Lambda^2}, \qquad \nu := \frac{m - \Lambda}{m + \Lambda} \in (0,1).
\]

\begin{theorem}[The one-dimensional case] \label{thm:explicit}
    Let $d = 1$ and $p > 1$. For all $\alpha > 0$, the optimal potential is unique up to translations, and is given by
    \begin{equation} \label{eq:explicit_V}
        V(x) = \frac{p}{p-1} \, \frac{m^2 - \lambda^2}{m + \Lambda \cosh \big( \frac{2 \kappa}{p-1} x \big)} \qquad \text{with} \qquad \lambda = \Lambda(\alpha).
    \end{equation}
    In particular, it is positive, even and decreasing on $\R^+$. The function $\alpha \mapsto \Lambda(\alpha)$ is the inverse of the function $\Lambda \mapsto \alpha(\Lambda) := \nrm{V}{L^p}$ on $(0, m)$, with
    \begin{equation} \label{eq:explicit_alpha}
        \alpha(\Lambda) = \frac{ (q+1)(m-\Lambda)}{(q \kappa)^{1/p}}
        \left( \frac{\sqrt{\pi}\,\Gamma(p)}{\Gamma\!\left(p+\tfrac12\right)}\,
        {}_2F_1\!\left(p,\tfrac12;\,p+\tfrac12;\,\nu\right) \right)^{1/p}.
    \end{equation}
\end{theorem}

It is interesting to compare with the Dirac case $\Dirac_m - V$ studied in~\cite[Theorem 1.2]{DolGonPiz-23}, where the optimal potential is
\[
    V(x) = \frac{p}{p-1} \, \frac{m^2 - \lambda^2}{m \cosh \big( \frac{2\kappa}{p-1} x \big) + \lambda}:
\]
the roles of $m$ and $\lambda$ in the denominator are exchanged.

\begin{proof}
    By Theorem~\ref{thm:existence}, there is an optimal potential $V$ for all $\alpha > 0$, with $\Lambda := \Lambda(\alpha) \in (0, m)$.

    \smallskip

    \noindent \textbf{Step 1: real reduction.} In dimension $d = 1$, the eigenvalues of $\Dirac_m - V\beta$ in $(-m, m)$ are simple (the solutions of the eigenvalue equation which decay at $+\infty$ form a one-dimensional space), and the anti-unitary operator $T = \sigma_3 \cK$ commutes with $\Dirac_m - V\beta$ and satisfies $T^2 = 1$ (see Proposition~\ref{prop:symmetries}). So the normalized eigenspinor $\psi := \mu^{-1/2}\phi$ of Theorem~\ref{thm:existence} can be chosen such that $T \psi = \psi$, that is of the form $\psi = (v, \ri u)^T$, with $u, v$ real-valued. With $f(s) := |s|^{q-1} s$, the Euler--Lagrange equation reads $V = f(v^2 - u^2)$, and the eigenvalue equation becomes
    \begin{equation} \label{eq:ODE_1d}
        \begin{pmatrix} u \\ v \end{pmatrix}' = - \begin{pmatrix} 0 & m - f(v^2 - u^2) - \Lambda \\ m - f(v^2 - u^2) + \Lambda & 0 \end{pmatrix} \begin{pmatrix} u \\ v \end{pmatrix}.
    \end{equation}
    Solutions to this system correspond to solitary waves of the one-dimensional nonlinear Dirac equation with Soler-type nonlinearity. In the cubic case $q = 1$, that is $p = 2$, this is the massive version of the Gross--Neveu model~\cite{GroNev-74}, whose solitary waves were computed explicitly in~\cite{LeeKuoGav-75} (see also~\cite{ChuPel-06} for explicit solitary waves of related coupled-mode systems). For general powers $q > 0$, explicit solitary waves can be found in~\cite{CooKhaMihSax-10, MerQuiCooKhaSax-12}.

    \smallskip

    \noindent \textbf{Step 2: uniqueness.} Let $G(s) := |s|^{q+1}/(q+1)$, so that $G' = f$. A direct computation shows that the energy
    \[
        E(u, v) := m(v^2 - u^2) - G(v^2 - u^2) - \Lambda (u^2 + v^2)
    \]
    is constant along solutions of~\eqref{eq:ODE_1d}. Since localized solutions tend to $0$ at infinity, they lie on the level set $\{E = 0\}$. With $s := v^2 - u^2$, this reads
    \[
        u^2 + v^2 = \frac{m s - G(s)}{\Lambda}.
    \]
    For $s < 0$, the right-hand side is negative, which is impossible. For $s \ge 0$, the constraint $u^2 + v^2 \ge s$ gives
    \[
        (m - \Lambda) s \ge G(s), \qquad \text{that is} \qquad 0 \le s \le s_{\max} := \big( (q+1)(m - \Lambda) \big)^{1/q}.
    \]
    So $\{E = 0\} \setminus \{0\}$ is the union of two loops, one in $\{v > 0\}$ and one in $\{v < 0\}$, which meet at the origin, and are exchanged by $(u,v) \mapsto (-u, -v)$. The trajectory of a localized solution is one of these homoclinic loops. Hence the localized solution is unique up to translations and sign, and $v^2 - u^2 \ge 0$ everywhere. In particular, $V = (v^2 - u^2)^q$ is positive.

    \medskip

    \noindent \textbf{Step 3: explicit formula.} One checks that the unique solution with $v$ even and $u$ odd is
    \[
        \psi(x) = \begin{pmatrix} v(x) \\ \ri u(x) \end{pmatrix} = \frac{1}{\sqrt{1-\nu T^2}} \left[(q+1)(m-\Lambda)\frac{1 - T^2}{1 - \nu T^2}\right]^{\frac1{2q}} \begin{pmatrix} 1 \\ \ri \sqrt{\nu} \, T \end{pmatrix},
    \]
    with $T := \tanh(q\kappa x)$.
    This gives
    \[
        V = f(v^2 - u^2) = (q+1)(m-\Lambda) \frac{1 - T^2}{1 - \nu T^2},
    \]
    which is~\eqref{eq:explicit_V}. Finally, with the change of variables $t = \tanh(q \kappa x)$ and then $t^2 = w$, we get
    \[
        \int_\R V^p = \frac{ \left((q+1)(m-\Lambda)\right)^{p}}{q \kappa} \int_{-1}^1 \frac{(1-t^2)^{p-1}}{(1-\nu t^2)^p} \, dt = \frac{ \left((q+1)(m-\Lambda)\right)^{p}}{q \kappa} \int_0^1 \frac{w^{-1/2}(1-w)^{p-1}}{(1-\nu w)^p} \, dw ,
    \]
    and Euler's integral representation of the hypergeometric function gives
    \[
        \int_0^1 \frac{w^{-1/2}(1-w)^{p-1}}{(1-\nu w)^p} \, dw = B\big(\tfrac12, p\big) \, {}_2F_1\big(p, \tfrac12; p + \tfrac12; \nu\big) \qquad \text{with} \qquad B\big(\tfrac12, p\big) = \frac{\sqrt{\pi}\,\Gamma(p)}{\Gamma(p + \frac12)}.
    \]
    This gives~\eqref{eq:explicit_alpha}. Since for each $\alpha > 0$ the optimal potential is of the form~\eqref{eq:explicit_V} with $\Lambda = \Lambda(\alpha)$ and $\nrm{V}{L^p} = \alpha$, we have $\alpha(\Lambda(\alpha)) = \alpha$.
\end{proof}

Inverting the exact formula~\eqref{eq:explicit_alpha} shows that $α ↦ \Lambda(\alpha)$ decays exponentially fast as $α → \infty$. We leave the details of the computation to the reader.

\begin{remark}[Jackiw--Rebbi kink]
    Recall that zero modes appear in dimension one only when $M := m -V$ has different signs at $\pm \infty$, see~\cite{JacReb-76}. This cannot happen here, since $M \to m$ at infinity. A sum of two kinks can be obtained in our setting by taking $V = 2m \1_{[-L, L]}$, (such that $\alpha = 2m L^{1/p}$). This $V$ produces eigenvalues approaching zero exponentially fast as $L \to \infty$, as shown e.g. in \cite{LuWatWei-20}.
\end{remark}

The function $\alpha \mapsto \Lambda(\alpha)$ is plotted in Figure~\ref{fig:Lambda_alpha} for several values of $p$. As $p \to \infty$, the curves approach $\Lambda(\alpha) = (m - \alpha)_+$, which is consistent with the bound $\lambda_1(V) \ge m - \nrm{V}{L^\infty}$ of Lemma~\ref{lem:bounds_lambda1}. As $p \to 1$, they approach $\Lambda(\alpha) = m / \cosh(\alpha)$. This limit can be understood with the following result, which is sharp.

\begin{proposition}[The case $p = d = 1$] \label{prop:p=1}
    Let $V \in L^1(\R, \R) \cap L^{p}(\R, \R)$ for some $p > 1$. Then all eigenvalues $\lambda \in (-m, m)$ of $\Dirac_m - V\beta$ satisfy
    \[
        \nrm{V_+}{L^1} \ge \operatorname{arccosh} \Big( \frac{m}{|\lambda|} \Big), \qquad \text{that is} \qquad |\lambda| \ge \frac{m}{\cosh(\nrm{V_+}{L^1})}.
    \]
    In particular, $\lambda_1(V) \ge m / \cosh(\nrm{V}{L^1})$. Equality is attained in the limit $V \to \alpha \delta_0$.
\end{proposition}

\begin{proof}
   By symmetry of the spectrum, we may assume $\lambda \in [0, m)$. As in Step~1 of the proof of Theorem~\ref{thm:explicit}, the eigenfunction can be written in the form
    \[
        \psi = \rho(x) \begin{pmatrix} \cos \theta(x) \\ \ri \sin \theta(x) \end{pmatrix} \qquad \text{with} \qquad \theta' = \lambda - (m - V) \cos(2 \theta),
    \]
    which can be obtained from the ODE system for the components as in \eqref{eq:ODE_1d}.
    Solutions approaching $0$ at $\pm \infty$ (where $V= 0$) shows that $\tan \theta(x) \to \mp \sqrt\nu$ as $x \to \mp \infty$, with $\nu = (m - \lambda)/(m + \lambda)$. Let $\theta_0 := \arctan \sqrt{\nu} \in (0, \pi/4]$, so that $\cos(2 \theta_0) = \lambda / m$, and choose the determination of $\theta$ such that $\theta(-\infty) = -\theta_0$. Then $\theta(+\infty) = \theta_0 + k \pi$ for some $k \in \Z$. On the lines $\{\cos(2\theta) = 0\}$, we have $\theta' = \lambda$, so $\theta$ can only cross them upwards since $\lambda > 0$.  Hence $\theta > -\pi/4$ everywhere, $k \ge 0$, and $\theta$ crosses the whole interval $(-\theta_0, \theta_0)$ upwards. On this interval, $\cos(2\theta) > \lambda/m$, so $$\theta' \le \lambda - m\cos(2\theta) + V_+ \cos(2 \theta) \le V_+ \cos(2\theta)$$ 
    hence
    \[
        \int_\R V_+ \ge \int_{-\theta_0}^{\theta_0} \frac{d\theta}{\cos (2 \theta)} = 2 \operatorname{artanh} (\tan \theta_0) = 2 \operatorname{artanh}(\sqrt{\nu}) = \ln \frac{m + \kappa}{\lambda} = \operatorname{arccosh}\Big( \frac{m}{\lambda} \Big),
    \]
    with $\kappa = (m^2 - \lambda^2)^{1/2}$. For $V = \alpha \delta_0$, the angle $\theta$ is constant equal to $\mp \theta_0$ on $\R_\mp$, and jumps across $0$ according to $d\theta/\cos(2\theta) = V dx$, so that the inequality is an equality. On the other hand, if $\theta$ takes values in $(-\theta_0, \theta_0)$ on a set of positive measure, the inequality is strict.
\end{proof}

\section{Numerical simulations in dimension \texorpdfstring{$d=2$}{d=2}}
\label{sec:appendix:numerics}

In this appendix, we study numerically the optimal potentials in dimension $d = 2$. We take $m = 1$ and, unless otherwise stated, $p = 3$. The code is available from the authors upon request.

\subsection*{The algorithm}
We follow the algorithm in~\cite{DolGonPiz-23}, which uses an iteration at the Birman-Schwinger level. Recall that we set, for $V ∈ L^p(\R^d, \R)$,
\[
    K_V(λ) := \sqrt{| V |} β \, {\rm sgn}(V) (\Dirac_m - \lambda)^{-1} \sqrt{| V |}.
\]
Since $K_{tV}(\lambda) = t K_V(\lambda)$ for all $t \in \R$, the Birman--Schwinger principle shows that $\lambda \in (-m,m)$ is an eigenvalue of $\Dirac_m - t V \beta$ if and only if $1/t$ is an eigenvalue of $K_V(\lambda)$. Introducing $\alpha(\lambda)$ the inverse map of $\Lambda(\alpha)$, we deduce as in~\cite{DolGonPiz-23} that 
\[
    \forall \lambda \in (0, m), \qquad \frac{1}{\alpha(\lambda)} = \max \Big\{ \mu_+ \big( K_W(\lambda) \big), \quad W \in L^p(\R^d, \R), \quad \nrm{W}{L^p} = 1 \Big\},
\]
where $\mu_+(K)$ denote the largest positive eigenvalue of the (compact) operator $K$, with $\mu_+(K) := 0$ if $K$ has no positive eigenvalue. In particular, $1 / \alpha_*(p,d) = \sup_W \mu_+(K_W(0))$. If $W$ is a maximizer, then
\[
    V := \frac{W}{\mu_+(K_W(\lambda))} \qquad \text{is an optimal potential for} \qquad \alpha = \alpha(\lambda) = \frac{1}{\mu_+(K_W(\lambda))}.
\]

\medskip

The main difference with respect to~\cite{DolGonPiz-23} is that the operator $K_V(\lambda)$ is not self-adjoint, so $\mu_+$ has no variational characterization. It is unclear that for a given $V$, we have $\mu_+(K_V) > 0$. It may happen that $K_V$ has no eigenvalue on the real line. But this will be the case whenever $λ$ is in the spectrum of $(\Dirac_m - t V \beta)$ for some $t > 0$. This happens in particular for the optimal potentials.We look for them with a self-consistent iteration. With $\cN(f) := f / \nrm{f}{L^p}$, we set
\[
    \begin{cases}
        \mu_k := \mu_+ \big( K_{W_k}(\lambda) \big), \qquad 
        \phi_k := (\Dirac_m - \lambda)^{-1} |W_k|^{1/2} \sign(W_k) \beta u_k, \\[4pt]
        W_{k+1} := \cN \Big( (1 - t) W_k + t \, \cN \big( | \sprod{\phi_k}{\beta\phi_k} |^{\frac{2-p}{p-1}} \sprod{\phi_k}{\beta\phi_k} \big) \Big),
    \end{cases}
\]
where $u_k$ is an eigenvector of $K_{W_k}(\lambda)$ for the eigenvalue $\mu_k$, and $t \in (0,1]$ is a damping parameter. We use $t = 1/2$ below. Since the problem is translation invariant, we also translate $W_{k+1}$ so that its center of mass is at the origin.

\medskip

All our simulations start from random initial data, and we can check that the sequence $(\mu_k)$ is increasing. In dimension $d = 1$, our algorithm reproduces the explicit solutions of Theorem~\ref{thm:explicit} up to numerical and discretization errors. All the critical points that we computed satisfy the virial identity of Proposition~\ref{prop:virial}, up to $10^{-5}$.

\medskip

We also numerically compute the optimal {\em radial} potential, which we denote by $V_{\rm rad}$. In order to compute this potential, we solve the radial equation using a shooting method. We take $\beta = \sigma_3$ and $\Dirac_m = \sigma_1 (-i \partial_1) + \sigma_2 (-i\partial_2) + m \sigma_3$, and we look for eigenfunctions of the form $\psi(r, \theta) = (v(r),  \ri u(r) e^{i \theta})^T$. This is the two-dimensional version of the Wakano ansatz~\cite{Wak-66} (see also~\cite{CueBouCom-18}), restricted to the lowest angular momentum channel. We checked numerically that, in the whole range of our simulations, this channel is the optimal one. In this channel, we have $\sprod{\psi}{\beta\psi} = v^2 - u^2$ is radial. With $M := m - V$, the equation $(\Dirac_m - V \beta) \psi = \lambda \psi$ becomes
\[
    \begin{cases}
        u' + \frac{u}{r} = (\lambda - M) v \\
        v' = - (\lambda + M) u.
    \end{cases}
\]
For a radial $W$ with $\nrm{W}{L^p} = 1$, $1/\mu_+(K_W(\lambda))$ is the smallest $t > 0$ for which this system with $V = t W$ has a solution which is regular at $r = 0$ and decays at infinity. We integrate the regular solution $(v, u)(0) = (1, 0)$ with a fourth-order Runge--Kutta scheme. The potential $V_{\rm rad}$ is then obtained with the same self-consistent iteration as above, among radial functions.

\medskip

\subsection*{Results} We see several interesting regimes which we describe below. The different regimes are described with the $λ ∈ (0, m = 1)$ variable, but could be translated in the $\alpha$ variable. Recall that the smaller the $\lambda$ the bigger the $\alpha$. We find the description with $\lambda$ easier to picture. In what follows, we fix $d = 2$ and $p = 3$.

\medskip

\noindent\underline{The optimal potentials are radial decreasing for $\lambda ∈ (0.126, 1)$}. We observe that random initial data always converge to a radial potential, which coincides with $V_{\rm rad}$. A typical run is displayed in Figure~\ref{fig:movie}. 

\begin{figure}[ht]
\centering
\includegraphics[width=\linewidth]{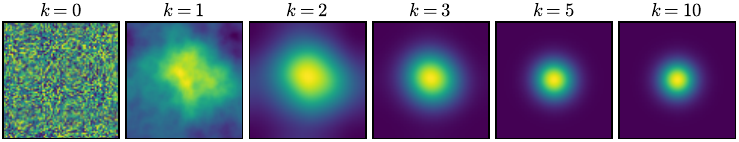}
\caption{The iterates $W_k$ for some values of $k$, in the case $d = 2$, $p = 3$ and $\lambda = 1/2$, starting from a purely random $W_0$.}
\label{fig:movie}
\end{figure}

The profiles of the radial critical points $V_{\rm rad}$ are displayed in Figure~\ref{fig:radial}. We observe that these potentials are positive, and radially decreasing for $\lambda \ge 0.126$. We also see that the effective mass $M = m - V$ takes negative values for $\lambda < \lambda_M \approx 0.55$.

\begin{figure}[ht]
\centering
\includegraphics[width=0.75\linewidth]{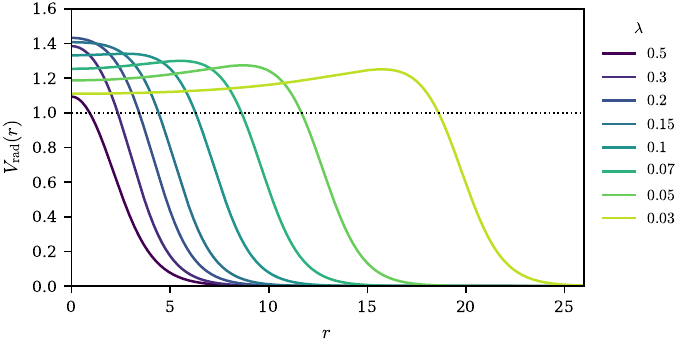}
\caption{The radial critical points $V_{\rm rad}$ for $d = 2$, $p = 3$ and several values of $\lambda$. The dotted line is $V = m$.}
\label{fig:radial}
\end{figure}

\medskip

\noindent\underline{The optimal potentials are radial, but not radially decreasing, for $0.104 < \lambda < 0.126$}. We then observe that the optimal radial potentials $V_{\rm rad}$ cease to be radial decreasing for $λ < 0.126$, see again Figure~\ref{fig:radial}. For $\lambda > 0.104$, the optimal potential still seems to be radial and positive. 

\medskip

\noindent\underline{The optimal potentials are no longer radial for $λ < 0.104$.} We display in Figure~\ref{fig:symmetry_breaking} the computed optimal potentials for some small values of $\lambda$. These potentials are obtained from several random initial data (and keeping the optimal result). We checked that most runs converge to similar potentials, with $L^p$ norms which agree up to $3 \cdot 10^{-4}$. We report in Table~\ref{tab:alpha_2d} the computed values of $\alpha(\lambda)$ and $\alpha_{\rm rad}(\lambda)$ for different values of $\lambda$.

\begin{table}[ht]
\centering
\begin{tabular}{|c|c|c|c|c|c|c|c|c|}
    \hline
    $\lambda$ & $0.1$ & $0.09$ & $0.08$ & $0.07$ & $0.06$ & $0.05$ & $0.04$ & $0.03$  \\
    \hline
    $\alpha_{\rm rad}(\lambda)$ & $6.5721$ & $6.9456$ & $7.3870$ & $7.9209$ & $8.5861$ & $9.4487$ & $10.6320$ & $12.4015$ \\
    \hline
    $\alpha(\lambda)$ & $6.5715$ & $6.9365$ & $7.3562$ & $7.8444$ & $8.4197$ & $9.1356$ & $10.0611$ & $11.3482$ \\
    \hline
\end{tabular}
\medskip
\caption{Comparison between $\alpha_{\rm rad}(\lambda)$ and $\alpha(\lambda)$ found numerically for $d = 2$ and $p = 3$.}
\label{tab:alpha_2d}
\end{table}

\medskip

This strongly indicates that the optimal potentials are \textbf{not radial} in this regime. Instead, they are concentrated on a flower-shaped region, whose number of petals increases as $\lambda$ decreases.

\begin{figure}[ht]
\centering
\includegraphics[width=\linewidth]{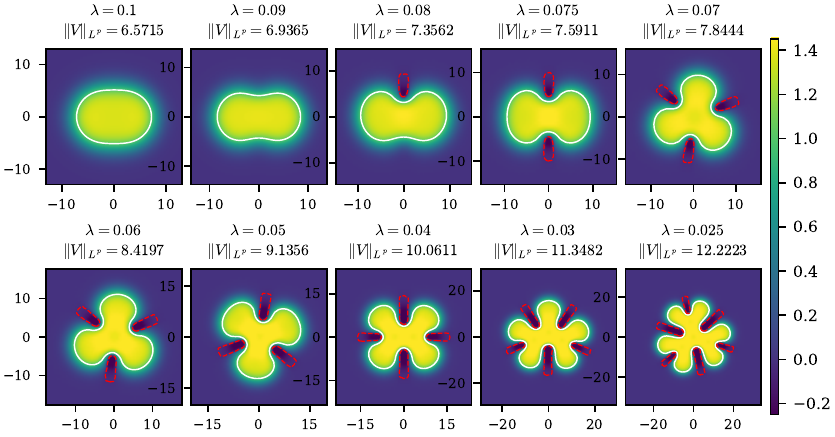}
\caption{Symmetry breaking in dimension $d = 2$. The best potentials $V$ found so far for $p = 3$ and for different values of $\lambda$ (note the different scales). The white line is the curve $\{V = m\}$, and the red lines surround the regions where $V < -0.02$.}
\label{fig:symmetry_breaking}
\end{figure}

\medskip

To compute a value of $\lambda$ under which the optimal potential cease to be optimal, we compute the stability of $V_{\rm rad}$ with respect to non-radial perturbation. More spectifically, 
we run the iteration from
\[
    W_0 := W_{\rm rad} \big( 1 + \eps \cos (n \theta) \big), \qquad \text{with} \qquad W_{\rm rad} := \frac{V_{\rm rad}}{\nrm{V_{\rm rad}}{L^p}}, \quad \eps = 10^{-3}, \quad n \in \{2, 3, 4\},
\] 
and find that $V_{\rm rad}$ becomes unstable in the direction $\cos(2\theta)$ for $\lambda < \lambda_c \approx 0.104$. In this regime, $V_{\rm rad}$ is not optimal, and the iteration started from random initial data converges to elongated critical points.

\medskip

\noindent\underline{The optimal potentials are no longer positive for $\lambda \le 0.08$.} Finally, it we continue to decrease the value of $\lambda$, we observe that the optimal potential take negative values for $\lambda \le 0.08$, see again Figure~\ref{fig:symmetry_breaking}. Interestingly, for $\lambda = 0.08$, we numerically find that, for the potential $V$ of Figure~\ref{fig:symmetry_breaking},
\[
    \lambda_1(V_+) \approx 0.080016 \qquad \text{and} \qquad \lambda_1(|V|) \approx 0.080029, \qquad \text{while} \qquad \lambda_1(V) = 0.08.
\]
In particular, the inequalities $\lambda_1(V_+) \le \lambda_1(V)$ and $\lambda_1(| V |) \le \lambda_1(V)$ fail for this particular $V$.

\section*{Acknowledgements}
The authors thank Jean Dolbeault and Fabio Pizzichillo for interesting discussions. The authors acknowledge support from ECOS 240024. H. VDB. received support from the Center for Mathematical Modeling (Universidad de Chile \& CNRS IRL 2807) through ANID/Basal project \# FB210005 and from ANID/Fondecyt project \# 125-0596.

\subsection*{Use of generative AI}
Claude Opus 5.5 was used to check and simplify some proofs, improve the bibliography, write most of the numerical code following the authors own algorithm used in~\cite{DolGonPiz-23}, help with the production of the images, and correct typos. No AI tools were used to produce mathematical results, nor to write the manuscript. 

\medskip

As we already mentioned in the manuscript, ChatGPT Astra claims to have a proof for Conjecture~\ref{thm:critical_alpha_dim3} (existence of a 3d potential generating a zero-mode). We do not yet understand its proof, and decided to put the result in a future work to keep this manuscript free of AI generated mathematical proofs.

\subsection*{Data availability} The code used for the numerical simulations of Appendix~\ref{sec:appendix:numerics}, and the corresponding data, are available from the authors upon request.

\subsection*{Conflict of interest}
The authors declare that they have no conflict of interest.

\printbibliography

\end{document}